\documentclass[twoside,leqno,10pt, A4]{amsart}
\usepackage{amsfonts}
\usepackage{amsmath}
\usepackage{amscd}
\usepackage{amssymb}
\usepackage{amsthm}
\usepackage{amsrefs}
\usepackage{latexsym}
\usepackage{mathrsfs}
\usepackage{bbm}
\usepackage{enumerate}
\usepackage{graphicx}
\usepackage{comment}

\usepackage{amsfonts}
\usepackage{amsmath}
\usepackage{amscd}
\usepackage{amssymb}
\usepackage{amsthm}
\usepackage{amsrefs}
\usepackage{latexsym}
\usepackage{mathrsfs}
\usepackage{bbm}
\usepackage{amscd}
\usepackage{amssymb}
\usepackage{amsthm}
\usepackage{amsrefs}
\usepackage{latexsym}
\usepackage{mathrsfs}
\usepackage{bbm}
\usepackage{enumerate}
\usepackage{graphicx}
\usepackage{color}
\begin{document}

\newtheorem{theorem}[subsection]{Theorem}
\newtheorem{proposition}[subsection]{Proposition}
\newtheorem{lemma}[subsection]{Lemma}
\newtheorem{corollary}[subsection]{Corollary}
\newtheorem{conjecture}[subsection]{Conjecture}
\newtheorem{prop}[subsection]{Proposition}
\numberwithin{equation}{section}
\newcommand{\mr}{\ensuremath{\mathbb R}}
\newcommand{\mc}{\ensuremath{\mathbb C}}
\newcommand{\dif}{\mathrm{d}}
\newcommand{\intz}{\mathbb{Z}}
\newcommand{\ratq}{\mathbb{Q}}
\newcommand{\natn}{\mathbb{N}}
\newcommand{\comc}{\mathbb{C}}
\newcommand{\rear}{\mathbb{R}}
\newcommand{\prip}{\mathbb{P}}
\newcommand{\uph}{\mathbb{H}}
\newcommand{\fief}{\mathbb{F}}
\newcommand{\majorarc}{\mathfrak{M}}
\newcommand{\minorarc}{\mathfrak{m}}
\newcommand{\sings}{\mathfrak{S}}
\newcommand{\fA}{\ensuremath{\mathfrak A}}
\newcommand{\mn}{\ensuremath{\mathbb N}}
\newcommand{\mq}{\ensuremath{\mathbb Q}}
\newcommand{\half}{\tfrac{1}{2}}
\newcommand{\f}{f\times \chi}
\newcommand{\summ}{\mathop{{\sum}^{\star}}}
\newcommand{\chiq}{\chi \bmod q}
\newcommand{\chidb}{\chi \bmod db}
\newcommand{\chid}{\chi \bmod d}
\newcommand{\sym}{\text{sym}^2}
\newcommand{\hhalf}{\tfrac{1}{2}}
\newcommand{\sumstar}{\sideset{}{^*}\sum}
\newcommand{\sumprime}{\sideset{}{'}\sum}
\newcommand{\sumprimeprime}{\sideset{}{''}\sum}
\newcommand{\sumflat}{\sideset{}{^\flat}\sum}
\newcommand{\shortmod}{\ensuremath{\negthickspace \negthickspace \negthickspace \pmod}}
\newcommand{\V}{V\left(\frac{nm}{q^2}\right)}
\newcommand{\sumi}{\mathop{{\sum}^{\dagger}}}
\newcommand{\mz}{\ensuremath{\mathbb Z}}
\newcommand{\leg}[2]{\left(\frac{#1}{#2}\right)}
\newcommand{\muK}{\mu_{\omega}}
\newcommand{\thalf}{\tfrac12}
\newcommand{\lp}{\left(}
\newcommand{\rp}{\right)}
\newcommand{\Lam}{\Lambda_{[i]}}
\newcommand{\lam}{\lambda}
\def\L{\fracwithdelims}
\def\om{\omega}
\def\pbar{\overline{\psi}}
\def\phi{\varphi}
\def\lam{\lambda}
\def\lbar{\overline{\lambda}}
\newcommand\Sum{\Cal S}
\def\Lam{\Lambda}
\newcommand{\sumtt}{\underset{(d,2)=1}{{\sum}^*}}
\newcommand{\sumt}{\underset{(d,2)=1}{\sum \nolimits^{*}} \widetilde w\left( \frac dX \right) }

\theoremstyle{plain}
\newtheorem{conj}{Conjecture}
\newtheorem{remark}[subsection]{Remark}

\makeatletter
\def\widebreve{\mathpalette\wide@breve}
\def\wide@breve#1#2{\sbox\z@{$#1#2$}%
     \mathop{\vbox{\m@th\ialign{##\crcr
\kern0.08em\brevefill#1{0.8\wd\z@}\crcr\noalign{\nointerlineskip}%
                    $\hss#1#2\hss$\crcr}}}\limits}
\def\brevefill#1#2{$\m@th\sbox\tw@{$#1($}%
  \hss\resizebox{#2}{\wd\tw@}{\rotatebox[origin=c]{90}{\upshape(}}\hss$}
\makeatletter

\title[Moments of quadratic Dirichlet $L$-functions over Function Fields]{Moments of quadratic Dirichlet $L$-functions over Function Fields}

%%\date{\today}
\author{Peng Gao and Liangyi Zhao}

\begin{abstract}
 In this paper, we establish the expected order of magnitude of the $k$th-moment of quadratic Dirichlet $L$-functions associated to
hyperelliptic curves of genus $g$ as well as of prime moduli over a fixed finite field $\mathbb{F}_{q}$ for all real $k \geq 0$.
\end{abstract}

\maketitle



\section{Introduction}
\label{sec 1}

The moments of central values of families of $L$-functions is a topic of great interest in analytical number theory, as they carry important arithmetic information. Due to the works of J. P. Keating and N. C. Snaith \cite{Keating-Snaith02} and of J. B.
Conrey, D. W. Farmer, J. P. Keating, M. O. Rubinstein and N. C. Snaith \cite{CFKRS}, we now have precise conjectures concerning the asymptotic behaviors of these moments.  Moreover, several systematic {\it modus operandi} towards establishing sharp upper and lower bounds for these moments of the conjectured order of magnitude are now available. For lower bounds, there are the works of Z. Rudnick and K. Soundararajan \cites{R&Sound, R&Sound1}, M. Radziwi{\l\l} and K. Soundararajan \cite{Radziwill&Sound1} and W. Heap and K. Soundararajan \cite{H&Sound}.  For upper bounds, the methods were developed by K. Soundararajan \cite{Sound2009} with a refinement by A. J. Harper \cite{Harper}, M. Radziwi{\l\l} and K. Soundararajan \cite{Radziwill&Sound}.  Note that the approach used in \cite{Sound2009} relies on the truth of the generalized Riemann hypothesis (GRH). \newline

  Among the many families of $L$-functions studied in the literature, that of quadratic Dirichlet $L$-functions received much attention.
Asymptotic formulas for this family at the central point were obtained by M. Jutila \cite{Jutila} for the first and the second moments, by K.
Soundararajan \cite{Sound2009} for the third moment and by Q. Shen \cite{Shen} for the fourth under GRH. For various improvements on the
error terms of these moments, see \cites{ViTa, DoHo, Young1, Young2, Sound2009, Sono}. In addition, combining the works of \cites{Harper, Sound2009, R&Sound, R&Sound1, Radziwill&Sound1, Gao2021-2, Gao2021-3}, sharp lower (resp. upper) bounds for the $k$-th moment of the above family have been established for all real $k \geq 0$ (resp. $0 \leq k \leq 2$) unconditionally, while sharp upper bounds have been established for all
real $k > 2$ under GRH. \newline

  In \cite{Jutila},  M. Jutila also studied the first moment of the family of quadratic Dirichlet $L$-functions with prime moduli, resolving a conjecture of D. Goldfeld and C. Viola \cite{G&V}.  S. Baluyot and K. Pratt \cite{BP} obtained an asymptotic formula for the second moment of the quadratic family of Dirichlet $L$-functions with prime moduli under GRH and also obtained sharp upper and lower bounds for the second moment unconditionally as well as for the third moment under certain assumptions. In \cite{G&Zhao11}, the authors obtained sharp upper and lower bounds for the $k$-th moment of quadratic Dirichlet $L$-functions averaged over prime moduli for all real $k \geq 0$, using the method developed by W. Heap and K. Soundararajan in \cite{H&Sound} for the lower bounds, and that by K. Soundararajan in \cite{Sound2009} together with its refinement by A. J. Harper in \cite{Harper} for the upper bounds. \newline

 In this paper, we consider the rational function field analogue of the above families of quadratic Dirichlet $L$-functions. To state our
results,  we fix a finite field $\mathbb{F}_{q}$ of cardinality $q \equiv 1 \pmod 4$ and denote by $A=\mathbb{F}_{q}[T]$ the
polynomial ring over $\mathbb{F}_{q}$. We reserve the symbol $P$ for a monic, irreducible polynomial in $A$ and we shall often refer to $P$ as a
prime in $A$. For a polynomial $f \in A$, we denote its degree by $\deg (f)$ and define the norm $|f|$ as $|f|=q^{\deg(f)}$ for
$f\neq 0$ and $|f|=0$ for $f=0$.  We set
\begin{align*}
%%\label{2.6}
\mathcal{H}_{2g+1,q}=& \{\text{$D\in A$ : square--free, monic and $\deg(D)=2g+1$}\} \quad \mbox{and} \\
\mathcal{P}_{2g+1,q}=& \{P \in A : P \ \text{ prime},  \deg(P)=2g+1 \}.
\end{align*}

  For $D \in \mathcal{H}_{2g+1,q}$, we write $\chi_{D}$ for the quadratic character defined in Section \ref{sec 2.1} and $L(s, \chi_D)$ for
the associated $L$-function. It is shown in \cite[Section 2]{Andrade&Keating12} that $L(s, \chi_D)$  is the numerator of the zeta
function associated to the hyperelliptic curve $C_D:y^2=D$ of genus $g$.  Similar to the classical setting, we are interested in the behavior
as $g \rightarrow \infty$ of the $k$-th moment of the quadratic families of $L$-functions at the central point given by
\begin{align}
\label{kthmoment}
\begin{split}
  & \sum_{D \in \mathcal{H}_{2g+1,q}} L ( \half, \chi_D ) ^k.
\end{split}
\end{align}

   An asymptotic formula for the first moment of the above family was proved by J. C. Andrade and J. P. Keating \cite{Andrade&Keating12}. An
extension of the result to general $q$ can be found in H. Jung \cite{Jung}. In \cite{Florea17}, A. Florea obtained a refinement of the main term
in the above mentioned result of \cite{Andrade&Keating12} for $q$ being prime. In \cites{Florea17-1, Florea17-2}, A. Florea further
established asymptotic formulas for the second, third and fourth moments of this family. A secondary term in the asymptotic formula of the third
moment was later found by A. Diaconu \cite{Diaconu19}, using multiple Dirichlet series. The results are all consistent with the conjectured
formulas for the integral moments of this family by J. C. Andrade and J. P. Keating \cite{Andrade&Keating14}. See \cite{RW} for a numerical
evidence of the conjectures and \cite{DT} for a refinement of the same. It is also interesting to consider the expression in \eqref{kthmoment} by restricting the sum over $D$ to $D \in \mathcal{H}_{2g+2,q}$ (similarly defined as $\mathcal{H}_{2g+1,q}$). Such
results can be found in \cites{Jung1, AM}. \newline

  In \cite{Andrade16}, J. C. Andrade applied the method of Z. Rudnick and K. Soundararajan in \cites{R&Sound, R&Sound1} to prove sharp
lower bounds for the $k$-th moment given in \eqref{kthmoment} for all odd $q$ and every even natural number $k$.  In the other direction, A. Florea obtained \cite[Theorem 2.7]{Florea17-2} upper bounds, optimal save for an $\varepsilon$ power of $g$, for the $k$-th moment given in
\eqref{kthmoment} for all prime numbers $q \equiv 1 \pmod 4$ and any real number $k>0$. \newline

  We are also interested in the behavior
as $g \rightarrow \infty$ of the $k$-th moment of this family of $L$-functions at the central point given by
\begin{equation}
\label{kthmoment1}
  \sum_{P \in \mathcal{P}_{2g+1,q}} L(\half, \chi_P)^k.
\end{equation}

J. Andrade and J. P. Keating \cite{Andrade&Keating} obtained asymptotic formulas for \eqref{kthmoment1} in the cases $k=1$ and $2$.  An improvement for the second moment result is given by H. M. Bui and A. Florea in \cite{B&F20}. Extending the recipe developed by J. B. Conrey, D. W. Farmer, J. P. Keating, M. O. Rubinstein, and N. C. Snaith \cite{CFKRS} to the function field case, J. C. Andrade,  H. Jung and A. Shamesaldeen \cite{AJS} conjectured that for every integer $k \geq 1$, there exists an polynomial $Q_k$, with an explicit expression, of degree $k(k+1)/2$ such that, as $g \rightarrow \infty$,
\begin{equation*}
%%\label{kthmomentasymp}
  \sum_{P \in \mathcal{P}_{2g+1,q}}L(\half, \chi_P)^k \sim \frac {X}{\log_q X}Q_k(\log_q X).
\end{equation*}
 Here and throughout the paper, we fix $X=q^{2g+1}$ and we shall also use the convention that when summing over polynomials in a subset $S \subset A$, the symbol $\sum_{f \in S}$ always stands for a summation over monic $f \in S$, unless otherwise specified, as most of the sums appearing in the paper are in fact about summing over ideals of $A$. \newline

 The aim of this paper is to extend the above results of Andrade and Florea on \eqref{kthmoment} to all real $k \geq 0$, and to establish the $k$-th moment of the family of $L$-functions given in \eqref{kthmoment1} to the conjectured order of magnitude.  Our main result gives the following optimal bounds for the $k$-th moment of
 \[ \{ L(1/2, \chi_D) : D \in \mathcal{H}_{2g+1, q}\} \quad \mbox{and} \quad \{ L(1/2, \chi_P) : P \in \mathcal{P}_{2g+1, q}\}, \quad \mbox{as} \quad g \rightarrow \infty. \]
\begin{theorem}
\label{thmorderofmag}
   Suppose that $q \equiv 1 \pmod 4$ is a prime number and $X=q^{2g+1}$. For every real number $k \geq 0$ we have for large $g$,
\begin{align*}
%%\label{orderofmag}
  \sum_{D\in\mathcal{H}_{2g+1,q}} L \left( \half,\chi_{D} \right)^{k} \asymp_k  X(\log_q X)^{k(k+1)/2} \quad \mbox{and} \quad
  \sum_{P \in\mathcal{P}_{2g+1,q}} L(\tfrac{1}{2},\chi_{P})^{k}  \asymp_k  X(\log_q X)^{k(k+1)/2-1}.
\end{align*}
\end{theorem}

We prove Theorem \ref{thmorderofmag} by establishing sharp lower and upper bounds for the moments.  Throughout the paper, we shall assume that $k>0$ when working with the $k$-th moment of $L$-functions since the case $k=0$ is trivial. For lower bounds,
we apply the lower bounds principle of  W. Heap and K. Soundararajan \cite{H&Sound}.  Note that this process requires an asymptotic formula for the twisted first moment with a power saving error term.   The formulas for the twisted first, second and third moment concerning $\{ L(1/2, \chi_D) : D \in \mathcal{H}_{2g+1, q}\}$ have already been obtained by H. M.  Bui and A. Florea \cite{BF18-1}.  Therefore, we only consider twisted moments concerning $\{ L(1/2, \chi_P) : P \in \mathcal{P}_{2g+1, q}\}$ here.  Our next result extends \cite[Theorem 2.4]{Andrade&Keating} to the twisted first moment.
\begin{proposition}
\label{Prop1'}
  With the notation as above and writing any monic $l \in A$ uniquely as $l=l_1l^2_2$ where $l_1, l_2$ are monic polynomials with $l_1$ monic and square-free, we have for any $\varepsilon>0$,
\begin{align}
\label{eq:1ndmomentP}
\begin{split}
 \sum_{P\in\mathcal{P}_{2g+1,q}} L(\tfrac{1}{2},\chi_{P})\chi_{P}(l)= \frac {X}{(\log_q X) \sqrt{|l_1|}} \Big (\log_q \Big ( \frac {\sqrt{X}}{|l_1|}\Big )+\frac 12 \Big )+O \left(X^{3/4+\varepsilon}|l|^{\varepsilon}\right ).
\end{split}
\end{align}
\end{proposition}

 We also include a result concerning the twisted second moment, extending \cite[Theorem 1.1]{B&F20}.
\begin{proposition}
\label{Prop1}
  With notations as in Proposition~\ref{Prop1'}, we have for any $\varepsilon>0$,
\begin{align}
\label{eq:1ndmomentE}
\begin{split}
 \sum_{P\in\mathcal{P}_{2g+1,q}} & L(\tfrac{1}{2}, \chi_{P})^2\chi_{P}(l)\\
  & =  \frac {X}{(\log_q X)} \frac {d_A(l_1)}{\sqrt{|l_1|}}\prod_{P|l_1} \left( 1+\frac 1{|P|} \right)^{-1} \\
& \hspace*{1cm} \times \left( \frac{1}{3\zeta_q(2)} \left( \frac 12\log_q \frac {X}{|l_1|} \right)^3+C(l) \left( \frac 12\log_q \frac {X}{|l_1|} \right)^2+ \left( \sum_{P|l_1}\frac {\deg(P)/|P|}{1+1/|P|} \right) \frac{1}{\zeta_q(2)} \left( \frac 12\log_q \frac {X}{|l_1|} \right)^2 \right) \\
& \hspace*{1cm} + O\left( |l|^{\varepsilon} \left( \frac 12\log_q \frac {X}{|l_1|} \right)^{3/2+\varepsilon}\frac {X}{\log_q X} \right),
\end{split}
\end{align}
  where $d_A=d_{2,A}$ is the divisor function defined in \eqref{ddef}, $C(l)=1+1/q$ when $\deg(l)$ is even and $C(l)=2$ otherwise.
\end{proposition}

   Our proof of Proposition \ref{Prop1} follows closely the approach in the proof of \cite[Theorem 1.1]{B&F20}. Notice that the error term appearing in \eqref{eq:1ndmomentE} is without power saving, so that the asymptotic formula given in \eqref{eq:1ndmomentE} is not strong enough for us to deploy in the lower bounds principle of W. Heap and K. Soundararajan to study the $k$-th moment of $\{ L(1/2, \chi_P) : P \in \mathcal{P}_{2g+1, q}\}$ (often needed for $0<k<1$).  We further note here that the $O$-term in \eqref{eq:1ndmomentE} can dominate over the main term if $\log g \gg \deg(l)$ does not hold.  Fortunately, the device developed by K. Soundararajan \cite{Sound2009} and its refinement by A. J. Harper \cite{Harper} concerning upper bounds of $L$-functions provides a resolution for this, since what one really needs here is an sharp upper bound for a mollified second moment.  Thus, we shall apply the Soundrarajan-Harper method to treat both the lower bounds for the $k$-th moment for $0<k<1$ and the upper bounds for all real $k \geq 0$ for both the families $\{ L(1/2, \chi_D) : D \in \mathcal{H}_{2g+1, q}\}$ and $\{ L(1/2, \chi_P) : P \in \mathcal{P}_{2g+1, q}\}$.  We further remark that even though the method of Soundararajan-Harper relies on GRH in the number fields case, our result in Theorem \ref{thmorderofmag} is unconditional since GRH is established over function fields. \newline

  We end this section by remarking that one may want to further extend our result with  mollifiers similar to the work of K. Soundararajan \cite{sound1} in the number field setting to establish a positive proportion non-vanishing result for the corresponding $L$-functions at the central point. In fact, much progress has been made in the function field setting.  In \cite{BF18}, H. M.  Bui and A. Florea computed the one-level density of the family $\{ L(1/2, \chi_D) | D \in \mathcal{H}_{2g+1, q}\}$ and showed that at least $94\%$ of the members of this family do not vanish. Using a geometric argument, W. Li \cite{Li18} showed that infinitely many of $L(1/2, \chi_D) = 0$.  Furthermore, J. S. Ellenberg, W. Li and M. Shusterman \cite{ELS} obtained an upper bound that tends to $0$ as $q \rightarrow \infty$ on the proportion of $L(1/2, \chi_D) = 0$ . Their result also implies that $L(1/2, \chi_P) \neq 0$ for any prime $P \in \mathcal{P}_{2g+1, q}$. % In \cite{AB}, J. C. Andrade and C. G. Best showed that for $100\%$ members of $D \in \mathcal{H}_{2g+1, q}$, $L(1/2, \chi_D) \neq 0$.

\section{Preliminaries}
\label{sec 2}

\subsection{Backgrounds on function fields}
\label{sec 2.1}

 In this section, we collect some basic facts about function fields, most of which can be found in \cite{Rosen02}. Recall that we write $A$ for the ring $\mathbb{F}_{q}[T]$, $\mathcal{M}$ for the set of monic polynomials in $A$, $\mathcal{M}_n$ for the set of
monic polynomials of degree $n$ in $A$ and $\mathcal{M}_{\leq n}$ the set of monic polynomials of degree less than or equal to $n$.
 Moreover, $P$ denotes a prime in $A$, by which we mean that $P$ is monic and irreducible.  Let $\pi_A(n)$ stand for the number of
such primes of degree $n$. We have the following results on these primes.
\begin{lemma}
\label{RS}
  We have
\begin{equation}
\label{ppt}
\pi_A(n) = \frac{q^n}{n}+O \Big(  \frac{q^{n/2}}{n} \Big).
\end{equation}
   Also, for $x \geq 2$, we have
\begin{align}
\label{logp}
\sum_{|P| \le x} \frac {\log |P|}{|P|} =& \log x + O(1)  \quad \mbox{and} \\
\label{lam2p}
\sum_{|P| \le x} \frac{1}{|P|} =& \log \log x + b+ O \left( \frac {1}{\log x} \right),
\end{align}
where $b$ is an absolute constant.
\end{lemma}
\begin{proof}
  The assertion in \eqref{ppt} is the Prime Polynomial Theorem \cite[Theorem 2.2]{Rosen02}. We apply \eqref{ppt} to see that
\begin{align*}
\sum_{|P| \le x} \frac{\log_q |P|}{|P|} = \sum_{n \le \log_q x} \frac{n \pi_A(n)}{q^n} = \sum_{n \le \log_q x} 1+O\Big(\sum_{n \le \log_q x}
\frac{1}{q^{n/2}} \Big)= \log_q x+O ( 1 ).
\end{align*}
Now \eqref{logp} follows easily from the above computation. Then \eqref{lam2p} is obtained using \eqref{logp} and the
approach used in the proof of \cite[Theorem 2.7, (d)]{MVa1}. This completes the proof of the lemma.
\end{proof}

  We define the zeta function $\zeta_{A}(s)$ associated to $A$ for $\Re(s)>1$ by
\begin{equation*}
%%\label{2.1}
\zeta_{A}(s)=\sum_{\substack{f\in A}}\frac{1}{|f|^{s}}=\prod_{P}(1-|P|^{-s})^{-1},
\end{equation*}
  where we recall our convention that the sum over $f$ is restricted to monic $f \in A$.  Since there are $q^n$ monic polynomials of degree
$n$, it follows that
\begin{equation*}
%%\label{2.2}
\zeta_{A}(s)=\frac{1}{1-q^{1-s}}.
\end{equation*}
  The above expression then defines $\zeta_{A}(s)$ on the entire complex plane with a simple pole at $s = 1$. We shall often write
$\zeta_{A}(s)=\mathcal{Z}(u)$ via a change of variables $u=q^{-s}$, yielding
\begin{equation*}
\mathcal{Z}(u)=\prod_{P}(1-u^{\deg(P)})^{-1}=(1-qu)^{-1}.
\end{equation*}

 We define the quadratic character $\leg {\cdot}{P}$ for any prime $P \in A$ by
\begin{equation*}
%%\label{2.4}
\leg {f}{P}=\Bigg\{
\begin{array}{cl}
0, & \mathrm{if}\ P \mid f,\\
1, & \mathrm{if}\ P \nmid f \ \mathrm{and} \ f \ \mathrm{is \ a \ square \ modulo} \ P,\\
-1, & \mathrm{if}\ P\nmid f \ \mathrm{and} \ f \ \mathrm{is \ a \ non\ square \ modulo} \ P.\\
\end{array}
\end{equation*}
The above definition can then be extended multiplicatively to $\leg {\cdot}{D}$ for any non-zero monic $D \in A$.
  The quadratic reciprocity law asserts that for non-zero monic $C, D \in A$ such that $(C, D)=1$, we have
\begin{equation*}
 \leg {C}{D}  = \leg{D}{C}(-1)^{\deg(C)\deg(D)(q-1)/2 }.
\end{equation*}

   We further define the quadratic character $\chi_{D}$ to be $\leg {D}{\cdot}$ and the $L$-function associated to
$\chi_{D}$ for $\Re(s)>1$ to be
\begin{equation*}
%%\label{2.5}
L(s,\chi_{D})=\sum_{\substack{f\in A}}\frac{\chi_{D}(f)}{|f|^{s}}=\prod_{P} \left( 1-\chi_{D}(P)|P|^{-s} \right)^{-1}.
\end{equation*}

  Similar to the function $\mathcal{Z}(u)$, we have via the change of variables $u=q^{-s}$,
$$ \mathcal{L}(u,\chi_D) = \sum_{f \in A} \chi_D(f) u^{\deg(f)} = \prod_P \left( 1-\chi_D(P) u^{\deg(P)} \right)^{-1}.$$

\subsection{Necessary Tools}

In this section, we gather some auxiliary results that will be used in the proof of our results. For a monic $f \in A$ and any integer $k \geq 1$, we define the $k$-th divisor function $d_{k,A}(f)$ by
\begin{equation} \label{ddef}
d_{k,A}(f) = \displaystyle \sum_{f_1 \cdot \ldots \cdot f_k =f} 1,
\end{equation}
where $f_i$ are monic polynomials in $A$. In particular, we note that $d_{1, A}(f)=1$ and $d_{2, A}(f)=d_A(f)$, the divisor function on $A$.
We have the following approximate functional equation, established in the proof of \cite[Lemma 2.1]{Florea17-2}.
 \begin{lemma}
\label{afe}
Let $u=e^{i \theta}$ with $\theta \in [0,2\pi)$. Then for any $D \in \mathcal{H}_{2g+1, q}$, we have for any integer $k \geq 1$,
 \begin{align*}
\mathcal{L} \Big( \frac{u}{\sqrt{q}},\chi_D \Big)^k &= \sum_{f \in \mathcal{M}_{\leq kg}} \frac{d_{A,k}(f)\chi_D(f)u^{\deg(f)} }{\sqrt{|f|}}+u^{2kg} \sum_{f \in \mathcal{M}_{\leq kg-1}} \frac{d_{A,k}(f)\chi_D(f) }{\sqrt{|f|}u^{\deg(f)}}.
 \end{align*}
 \end{lemma}

   Our next result is a direct consequence of \cite[Theorem 2.4]{BF18-1} on twisted first moment of $\{ L(1/2, \chi_D) : D \in \mathcal{H}_{2g+1, q}\}$.
\begin{proposition}
\label{Prop-twisted1stmoment}
  Writing any monic $l \in A$ uniquely as $l=l_1l^2_2$ where $l_1, l_2$ are monic polynomials with $l_1$ monic and
square-free, we have for any $\varepsilon>0$,
\begin{align*}
%%\label{eq:1ndmoment}
\begin{split}
 \sum_{D\in\mathcal{H}_{2g+1,q}}L \left( \frac{1}{2},\chi_{D} \right)\chi_{D}(l)
=&  \frac{C }{\zeta_A(2)}\frac {X}{\sqrt{|l_1|}g(l)} \Big (\log_q \Big ( \frac {\sqrt{X}}{|l_1|}\Big )+C_1+\sum_{\substack{P | l}} \frac
{C_1(P)}{|P|}\log_q |P| \Big )+O \left(X^{7/8+\varepsilon}|l|^{1/4+\varepsilon}\right ),
\end{split}
\end{align*}
where $C_1$ is an absolute constant,
\[ C= \displaystyle \prod_{\substack{P}}\left (1-\frac 1{|P|(|P|+1)} \right ), \; g(l)=\displaystyle \prod_{P | l}\Big ( \frac {|P|+1}{|P|} \Big )\left (1-\frac 1{|P|(|P|+1)} \right ) \; \mbox{and} \; C_1(P) \ll \frac{1}{|P|} \]
for all $P$.
\end{proposition}

The $O$-term in Proposition~\ref{Prop-twisted1stmoment} is weaker than that in \cite[Theorem 2.4]{BF18-1}.  But it is sufficient for our purpose here.  Now writing $f=\square$ if $f \in A$ is a perfect square, we have the following character sum results given in \cite[Lemma 4.1]{Andrade16}.
\begin{lemma}
\label{lem3}
Writting $X=q^{2g+1}$, we have
\begin{equation*}
%%\label{4.9}
\sum_{\substack{D\in\mathcal{H}_{2g+1,q} }}\left(\frac{D}{f}\right)=\delta_{f=\square}\Big (\frac{X}{\zeta_{A}(2)} \Big )\prod_{\substack{P\mid
f}}\left(\frac{|P|}{|P|+1}\right)+O\left(X^{1/2}|f|^{1/4}\right).
\end{equation*}
\end{lemma}

   We also note the following result for character sums over primes.
\begin{lemma}
\label{PVlemma}
  Recall that $X=q^{2g+1}$. We have, for $f$ not a perfect square,
\begin{align}
\label{Psumfnonsquare}
\sum_{ P \in \mathcal{P}_{2g+1,q}} \chi_P(f) \ll \frac {X^{1/2}}{\log_q X}  \deg(f).
\end{align}
 Also, for any $\varepsilon>0$,
\begin{align}
\label{wsum}
\sum_{\substack{P\in\mathcal{P}_{2g+1,q}}}\left(\frac{P}{f }\right)=  \delta_{f=\square}\Big (\frac {X}{\log_q X}\Big )+O \left( X^{1/2+\varepsilon}|f|^{\varepsilon} \right).
\end{align}
 \end{lemma}
\begin{proof}
   The estimation given in \eqref{Psumfnonsquare} is essentially the Weil bound established in \cite[(2.5)]{Rudnick10}, by taking also \eqref{ppt} into account. The estimation given in \eqref{wsum} is an easy consequence of \eqref{ppt} and \eqref{Psumfnonsquare}.
\end{proof}

   We have the following analogue of Perron’s formula in function fields, which is an easy consequence of Cauchy's residue theorem  (see \cite[(2.6)]{Florea17-2}).
\begin{lemma}
\label{lem4}
 Suppose that the power series $\sum^{\infty}_{n=0} a(n)u^n$ is absolutely convergent in $|u| \leq  r < 1$, then for $N \geq 0$,
\begin{align}
\label{perron1}
\begin{split}
 \sum_{n \leq N}a(n)=& \frac 1{2\pi i}\oint\limits_{|u|=r} \Big ( \sum^{\infty}_{n=0}a(n)u^n\Big )\frac {\dif u}{(1-u)u^{N+1}}.
\end{split}
\end{align}
\end{lemma}

   We end this section by including an expression for the twisted second moment of $\{ L(1/2, \chi_P) : P \in \mathcal{P}_{2g+1, q}\}$.
\begin{lemma}
\label{lem:Ltwistsbound}
 For $|u|=1$, we have
\begin{align}
\label{Lbounds}
\begin{split}
 \sum_{P \in \mathcal{P}_{2g+1,q}}  \mathcal{L} \Big( \frac{u}{\sqrt{q}}, \chi_P \Big)^2\chi_P(l)  =  \frac {X}{\log_q X} & \Big (\sum_{\substack{f \in \mathcal{M}_{\leq 2g} \\ fl=\square}} \frac{ d_{A}(f)u^{\deg(f)}}{\sqrt{|f|}}  + u^{4g} \sum_{\substack{f \in \mathcal{M}_{\leq 2g-1} \\ fl=\square}} \frac{ d_{A}(f)}{\sqrt{|f|} u^{\deg(f)}}\Big )\\
&+O \left( \frac {g(g+\deg(l))X}{\log_q X} + |l|^{\varepsilon}X^{1/2+\varepsilon} \right).
\end{split}
\end{align}
 \end{lemma}
\begin{proof}
Lemmas~\ref{afe} and \ref{PVlemma} give that, for $|u|=1$,
\begin{align}
\label{L2exp}
\begin{split}
 \sum_{P \in \mathcal{P}_{2g+1,q}}  \mathcal{L} \Big( \frac{u}{\sqrt{q}}, \chi_P \Big)^2\chi_P(l)  =  \frac {X}{\log_q X} & \Big (\sum_{\substack{f \in \mathcal{M}_{\leq 2g} \\ fl=\square}} \frac{ d_{A}(f)u^{\deg(f)}}{\sqrt{|f|}}  + u^{4g} \sum_{\substack{f \in \mathcal{M}_{\leq 2g-1} \\ fl=\square}} \frac{ d_{A}(f)}{\sqrt{|f|} u^{ \deg(f)}}\Big )\\
&+O\Big( \frac {X^{1/2}}{\log_q X} \Big( \sum_{\substack{f \in \mathcal{M}_{\leq 2g}}} \frac{ d_{A}(f)\deg(fl)}{\sqrt{|f|}} + \sum_{f \in \mathcal{M}_{\leq 2g-1}} \frac{ d_{A}(f)\deg(fl)}{\sqrt{|f|}} \Big) \Big).
\end{split}
 \end{align}

  Now, we note the following estimation from \cite[Lemma 4.2]{Andrade&Keating}:
\begin{align*}
 \sum_{\substack{f \in \mathcal{M}_{n}}}d_A(f) \ll nq^n.
\end{align*}

  Using the above and the observation that $\deg(fl)=\deg(f)+\deg(l)$, we get that
\begin{align}
\label{nonsquareerr}
 \sum_{\substack{f \in \mathcal{M}_{\leq 2g}}} \frac{ d_{A}(f)\deg(fl)}{\sqrt{|f|}} =\sum^{2g}_{n=0}(n+\deg(l))q^{-n/2} \sum_{\substack{f \in \mathcal{M}_{n}}}d_A(f)) \ll \sum^{2g}_{n=0}(n+\deg(l))q^{-n/2}nq^n \ll (g^2+g \deg(l))X^{1/2},
 \end{align}

  Similarly, we have
\begin{align}
\label{squareerr}
 \sum_{\substack{f \in \mathcal{M}_{\leq 2g}}} \frac{ d_{A}(f)d_A(fl)}{\sqrt{|f|}} \ll \sum_{\substack{f \in \mathcal{M}_{\leq 2g}}} \frac{ d_{A}(f)(|f|^{\varepsilon} |l|^{\varepsilon})}{\sqrt{|f|}} \ll |l|^{\varepsilon}X^{1/2+\varepsilon}.
 \end{align}
Inserting the bounds in \eqref{nonsquareerr} and \eqref{squareerr} into \eqref{L2exp} now leads to \eqref{Lbounds} and completes the proof.
\end{proof}

\subsection{Bounds for $L$-functions}

  In this section, we gather various upper bounds concerning the $L$-functions under consideration in this paper.
Note that for $D \in \mathcal{H}_{2g+1,q}$, we have, as shown in \cite[(2.2)]{Florea17-2},
\begin{equation*}
%%\label{produs}
 \mathcal{L}(u,\chi_D) = \prod_{j=1}^{2g} \left( 1-\alpha_j q^{1/2}u \right) .
\end{equation*}
 where $|\alpha_j|=1$ by the Riemann hypothesis proven by A. Weil \cite{Weil}. Moreover, it is known that
\begin{equation} \label{Lpostive}
 L\left( \tfrac 12,\chi_D \right) \geq 0.
\end{equation}
We quote \cite[Corollary 4.3]{B&F20} (by further noting \eqref{ppt}) for the following upper bound concerning the second moment of quadratic $L$-functions.
\begin{lemma}
\label{corollaryupper}
Let $u=e^{i \theta}$ with $\theta \in [0,2\pi)$. Then
 \begin{equation*}
  \sum_{P \in \mathcal{P}_{2g+1}} \Big|\mathcal{L} \Big(\frac{u}{\sqrt{q}},\chi_P \Big) \Big|^2 \ll_\varepsilon \Big (\frac {X}{\log_q X} \Big )g^{1+\varepsilon} \min \Big\{ g,\frac{1}{ \overline{2\theta}}\Big\}^{2}.
 \end{equation*}
 Here $\overline{\theta} = \min \{ \theta, 2 \pi - \theta \}$ for $\theta \in [0, 2\pi)$.
 \end{lemma}

    We notice also the following special cases corresponding to $u=1$ in \cite[Theorem 2.7]{Florea17-2} and \cite[Proposition 4.1]{B&F20} concerning upper bounds for the $k$-th moment of $L$-functions.
\begin{lemma}
\label{prop: upperbound}
 For any real number $n \geq 0$ and any $\varepsilon>0$, we have
\begin{align}
\label{upperboundD}
    \sum_{D \in \mathcal{H}_{2g+1,q}}  L  \left( \half, \chi_D \right)^{2n} \ll & X (\log_q X)^{n(2n+1)+\varepsilon} \quad \mbox{and} \\
\label{upperboundP}
     \sum_{P \in \mathcal{P}_{2g+1,q}} L  \left( \half, \chi_P \right)^{2n} \ll & X (\log_q X)^{n(2n+1)-1+\varepsilon}.
\end{align}
\end{lemma}

In applying the upper bounds method of Soundararajan-Harper, we shall also make crucial use of an upper bound on $\log |L(1/2, \chi_{D})|$ for $D \in \mathcal{H}_{2g+1, q}$ in terms of a short Dirichlet polynomial over the primes. We now derive such a bound basing on the devices developed in the proofs of \cite[Theorem 3.3]{AT14}, \cite[Proposition 4.3]{BFK} and \cite[Lemma 3.1]{DFL}. We recall from the above that $L(s,\chi_{D})$  is a polynomial of degree at most $2g$ for $D \in H_{2g+1, q}$.  This allows us to
proceed as in the proof of \cite[Proposition 4.3]{BFK} by setting $m=2g$ there. Upon making use of the proof of \cite[Theorem 3.3]{AT14} as
well, we arrive at the following result that is analogue to \cite[Proposition 4.3]{BFK}.
\begin{proposition}
\label{prop-ub}
Let $D \in \mathcal{H}_{2g+1, q}$ and $m =2g$.  For $h \leq m$ and $z$ with $\Re(z) \geq 0$,  we have
\begin{align}
\label{logLupperbound}
\log \big| L(\tfrac{1}{2}+z, \chi_{D}) \big| \leq \frac{m}{h} + \frac{1}{h} \Re \bigg(  \sum_{\substack{j \geq 1 \\ d(P^j) \leq h}} \frac{
\chi_{D}(P^j) \log q^{h - j \deg(P)}}{|P|^{j ( 1/2+z+1/(h \log q) )} \log q^j} \bigg).
\end{align}
\end{proposition}

 Arguing as in the proof of \cite[Lemma 3.1]{DFL}, we see that the terms $j \geq 3$ in \eqref{logLupperbound} contribute $O(1)$. We further
set $m=2g$, $z=0$, $X=q^{2g+1}$, $x=q^h$ and estimate the term $j=2$ using Lemma \ref{RS} to arrive from Proposition \ref{prop-ub} the following upper bound for $\log | {\textstyle L(1/2,\chi_D) }|$.
\begin{lemma}
\label{prop-logup}
Let $D \in \mathcal{H}_{2g+1, q}$. We have for any $x \leq X$,
\begin{align*}
%%\label{logLupperbound1}
\log | {\textstyle L(\frac{1}{2},\chi_D) }| & \leq   \sum_{\substack{ |P| \leq x}} \frac{ \chi_{D}(P) }{|P|^{1/2+1/\log x}
} \frac {\log (x/|P|)}{\log x}+ \frac 12 \log \log x+\frac {\log X}{\log x}+O(1).
\end{align*}
\end{lemma}

The above lemma is reminiscent of \cite[Proposition 3]{Harper}, a result on the classical quadratic Dirichlet $L$-functions.

\section{Proof of Proposition \ref{Prop1'}}

  Setting $k=u=1$ in Lemma \ref{afe} implies that
\begin{align}
\label{4.30}
\sum_{P \in\mathcal{P}_{2g+1,q}}L(\tfrac{1}{2},\chi_{P})\chi_{P}(l)
&=\sum_{\deg(f_{1})\leq g}\frac{1}{\sqrt{|f_{1}|}}\sum_{P \in\mathcal{P}_{2g+1,q}}\left(\frac{P}{f_{1}l}\right)+\sum_{\deg(f_{2})\leq
g-1}\frac{1}{\sqrt{|f_{2}|}}\sum_{P \in\mathcal{P}_{2g+1,q}}\left(\frac{P}{f_{2}l}\right).
\end{align}
If $f_{1}l \neq \square$ or $f_{2}l \neq \square$, then Lemma \ref{PVlemma} yields, keeping in mind that $X = q^{2g+1}$,
\begin{align}
\label{4.31}
\begin{split}
\sum_{\deg(f_{1})\leq g}\frac{1}{\sqrt{|f_{1}|}}\sum_{P \in\mathcal{P}_{2g+1,q}}\left(\frac{P}{f_{1}l}\right)
\ll & \sum_{\deg(f_{1})\leq g}\frac{|f_{1}l|^{\varepsilon}}{\sqrt{|f_{1}|}}X^{1/2+\varepsilon}
\ll
X^{3/4+\varepsilon}|l|^{\varepsilon},  \quad \mbox{or} \\
\sum_{\deg(f_{2})\leq g-1}\frac{1}{\sqrt{|f_{2}|}}\sum_{P \in\mathcal{P}_{2g+1,q}}\left(\frac{P}{f_{2}l}\right)
\ll &
X^{3/4+\varepsilon}|l|^{\varepsilon}.
\end{split}
\end{align}

Note that, similar to the integer case given in \cite[Theorem 2.11]{MVa}),  we have for any $\varepsilon >0$,
\begin{align}
\label{dbound}
\begin{split}
 d_A(f) \ll |f|^{\varepsilon}.
\end{split}
\end{align}

  If $f_{1}l=\square$, then $f_{1}$ is of the form $l_1f^{2}$.  We apply \eqref{ppt} and \eqref{dbound} to see that
\begin{align}
\begin{split}
\label{4.35}
\sum_{\substack{\deg(f_{1})\leq g \\ f_{1}l=\square}}\frac{1}{\sqrt{|f_{1}|}}\sum_{P \in\mathcal{P}_{2g+1,q}}\left(\frac{P}{f_{1}l}\right)
= & \frac {X}{\log_q X} \sum_{\substack{\deg(f_{1})\leq g \\ f_{1}l=\square}}\frac{1}{\sqrt{|f_{1}|}}+O\Bigg( \sqrt{X}\sum_{\substack{\deg(f_{1})\leq g
}}\frac{d_A(lf_1)}{\sqrt{|f_{1}|}}\Bigg) \\
=& \frac{X}{\sqrt{|l_{1}|} \log_q X }\sum_{\substack{
\deg(f)\leq (g-\deg(l_1))/2}}\frac{1}{|f|}+O\left( X^{3/4+\varepsilon}|l|^{\varepsilon} \right).
\end{split}
\end{align}

To evaluate the last sum above, we note that there are $q^n$ monic polynomials of degree $n$.  If $g-\deg(l_1)$ is even, we get
\[ \sum_{\deg(f)\leq (g-\deg(l_1))/2}\frac{1}{|f|} = \frac{g-\deg(l_1)}{2} +1 = \frac{1}{2} \log_q \frac{\sqrt{X}}{|l_1|} + \frac{3}{4} . \]
A similar treatment for the sum over $f_2$ with $f_2l= \square$ leads to, recalling that $g-\deg(l_1)$ is even,
\[ \sum_{\deg(f)\leq (g-\deg(l_1))/2}\frac{1}{|f|} = \frac{g-\deg(l_1)-2}{2} +1 = \frac{1}{2} \log_q \frac{\sqrt{X}}{|l_1|} - \frac{1}{4} . \]
The above two expressions lead to \eqref{eq:1ndmomentP} if $g-\deg(l_1)$ is even.  The case of add $g-\deg(l_1)$ is proved in using similar arguments. This completes the proof.

\section{Proof of Proposition \ref{Prop1}}

Now setting $k=2$, $u=1$ in Lemma \ref{afe} yields
\[  S:= \sum_{P\in \mathcal{P}_{2g+1,q}} L \big( \tfrac{1}{2},\chi_P \big)^2\chi_P(l)=\sum_{P\in \mathcal{P}_{2g+1,q}}\sum_{f \in \mathcal{M}_{\leq 2g}} \frac{d_{A}(f)\chi_P(lf)}{\sqrt{|f|}}+ \sum_{P\in \mathcal{P}_{2g+1,q}}\sum_{f \in \mathcal{M}_{\leq 2g-1}} \frac{d_{A}(f)\chi_P(lf)}{\sqrt{|f|}}. \]

   We apply Lemma \ref{lem4} to see that for $0<r<1$ and some integer $Y \leq g$,
\[ S =S_1(Y)+S_2(Y), \]
where
\[ S_1 (Y) = \sum_{P\in \mathcal{P}_{2g+1,q}}\sum_{f \in \mathcal{M}_{\leq 2Y}} \frac{d_{A}(f)\chi_P(lf)}{\sqrt{|f|}}+ \sum_{P\in \mathcal{P}_{2g+1,q}}\sum_{f \in \mathcal{M}_{\leq 2Y-1}} \frac{d_{A}(f)\chi_P(lf)}{\sqrt{|f|}} \]
and
\[ S_2(Y) = \frac{1}{2 \pi i} \oint\limits_{|u|=r}\sum_{P \in \mathcal{P}_{2g+1,q}}  \mathcal{L} \Big( \frac{u}{\sqrt{q}}, \chi_P \Big)^2\chi_P(l) \, \frac{(1+u)(1-u^{2g-2Y}) \dif u}{u^{2g+1}(1-u)} . \]

\subsection{Evaluation of $S_2(Y)$}

   We evaluate $S_2(Y)$ in this section by noting that the integrand in the expression for $S_2(Y)$ is analytical when $|u|=1$ since
\[
\Big|\frac{1-u^{2g-2Y}}{1-u}\Big|\leq 2 (g-Y).
\]

 We may thus shift the contour of integration to $|u|=1$ in the expression for $S_2(Y)$ and apply Lemma \ref{lem4}, getting
\begin{align*}
S_2(Y) =& \frac{1}{2 \pi i} \oint\limits_{|u|=1}\sum_{P \in \mathcal{P}_{2g+1,q}}  \mathcal{L} \Big( \frac{u}{\sqrt{q}}, \chi_P \Big)^2\chi_P(l) \, \frac{(1-u^{2g-2Y})(1+u) \, \dif u}{u^{2g+1}(1-u)} \\
=& \frac{1}{2 \pi i}\int\limits_{C_{1}} \sum_{P \in \mathcal{P}_{2g+1,q}}  \mathcal{L} \Big( \frac{u}{\sqrt{q}}, \chi_P \Big)^2\chi_P(l) \, \frac{(1-u^{2g-2Y})(1+u) \, \dif u}{u^{2g+1}(1-u)} \\
& +\ \frac{1}{2 \pi i}\int\limits_{C_{2}} \sum_{P \in \mathcal{P}_{2g+1,q}}  \mathcal{L} \Big( \frac{u}{\sqrt{q}}, \chi_P \Big)^2\chi_P(l) \, \frac{(1-u^{2g-2Y})(1+u)\, \dif u}{u^{2g+1}(1-u)} :=S_{21}(Y) + S_{22}(Y),
 \end{align*}
say.   Here $C_{1}$ denotes the arc of angle $4\pi\theta_1$ centered around $1$,  with $1/g\ll \theta_1=o(1)$, and $C_{2}$ is the complement. Lemma \ref{corollaryupper} can be used to bound $S_{22}$, giving
\begin{equation}
\label{E2}
S_{22}(Y) \ll_\varepsilon \frac {X}{\log_q X}  g^{1+\varepsilon} (g-Y)\theta_1^{-1}.
\end{equation}

  To estimate $S_{21}(Y)$, we apply Lemma \ref{lem:Ltwistsbound} and notice that when $fl=\square$, then $f$ is of the form $l_1f^{2}$. We then deduce that
 \begin{align*}
 S_{21}(Y) =& \frac {X}{\log_q X} \sum_{f \in \mathcal{M}_{\leq (2g-\deg(l_1))/2}} \frac 1{\sqrt{|l_1|}}\frac{ d_{A}(l_1f^2)}{|f|}\frac{1}{2 \pi i}\int\limits_{C_{1}}\frac{(1-u^{2g-2Y})(1+u)\, \dif u}{u^{2g-2\deg(f)-\deg(l_1)+1}(1-u)}\\
 &\hspace*{2cm} + \frac {X}{\log_q X} \sum_{f \in \mathcal{M}_{\leq (2g-\deg(l_1)-1)/2}}\frac 1{\sqrt{|l_1|}} \frac{ d_{A}(l_1f^2)}{|f|} \frac{1}{2 \pi i}\int\limits_{C_{1}}\frac{(1-u^{2g-2Y})(1+u)\, \dif u}{u^{-2g+2 \deg(f)+\deg(l_1)+1}(1-u)} \\
& \hspace*{3cm} +O(g(g+\deg(l))(g-Y)\theta_1+|l|^{\varepsilon}X^{1/2+\varepsilon}(g-Y)\theta_1).
 \end{align*}

   Note that
 \begin{equation*}
%%\label{series}
\sum_{f \in \mathcal{M}} d_{A}(l_1f^2) v^{\deg(f)} =\prod_{P |l_1} \left( \sum^{\infty}_{n=0}(2n+2)v^{n\deg(P)} \right) \prod_{P \nmid l_1} \left( \sum^{\infty}_{n=0}(2n+1)v^{n\deg(P)} \right)=\deg(l_1)C(v;l_1) \frac{ \mathcal{Z}(v)^3}{\mathcal{Z}(v^2)},
\end{equation*}
  where
\begin{equation*}
%%\label{Cdef}
C(v;l_1)=\prod_{P|l_1} \left( 1+v^{\deg(P)} \right)^{-1}.
\end{equation*}

  We now apply \eqref{perron1} to deduce that when $\deg(l_1)$ is even and for $r<1$,
 \begin{align*}
 S_{21}(Y)=& \frac {X}{\log_q X}  \frac {d_A(l_1)}{\sqrt{|l_1|}}\frac{1}{(2 \pi i)^2}\oint\limits_{|v|=r}\int\limits_{C_{1}}\frac{ C(u^2v/q;l_1)\mathcal{Z}(u^2v/q)^3 (1-u^{2g-2Y})(1+u)}{\mathcal{Z}(u^4v^2/q^2) (1-u)(1-v) u^{2g-\deg(l_1)+1} v^{g-\deg(l_1)/2+1}} \, \dif u \, \dif v \\
 & \hspace*{2cm}+ \frac {X}{\log_q X} \frac {d_A(l_1)}{\sqrt{|l_1|}}\frac{1}{(2 \pi i)^2}\oint\limits_{|v|=r}\int\limits_{C_{1}}\frac{ C(v/(u^2q);l_1) \mathcal{Z}(v/(u^2q))^3(1-u^{2g-2Y})(1+u)}{\mathcal{Z}(v^2/(u^4q^2)) (1-u) (1-v)u^{-2g+\deg(l_1)+1}v^{g-\deg(l_1)/2}} \, \dif u \, \dif v \\
& \hspace*{3cm}+O \left( \frac {g(g+\deg(l))(g-Y)X\theta_1}{\log_q X}+|l|^{\varepsilon}X^{1/2+\varepsilon}(g-Y)\theta_1 \right) \\
=& \frac {X}{\log_q X} \frac {d_A(l_1)}{\sqrt{|l_1|}} \frac{1}{(2 \pi i)^2}\oint\limits_{|v|=r}\int\limits_{C_{1}}\frac{C(v/q;l_1)(1-u^{2g-2Y})(1-v^2/q)(1+u)}{uv^{g-\deg(l_1)/2+1}(1-u)(1-v/u^2)(1-v)^3} \,\dif u \, \dif v \\
 &\hspace*{2cm}+\frac {X}{\log_q X} \frac {d_A(l_1)}{\sqrt{|l_1|}}\frac{1}{(2 \pi i)^2}\oint\limits_{|v|=r}\int\limits_{C_{1}}\frac{u C(v/q;l_1)(1-u^{2g-2Y})(1-v^2/q)(1+u)}{v^{g-\deg(l_1)/2}(1-u)(1-u^2v)(1-v)^3} \,\dif u \, \dif v \\
&\hspace*{3cm} +O\left( \frac {g(g+\deg(l))(g-Y)X\theta_1}{\log_q X}+|l|^{\varepsilon}X^{1/2+\varepsilon}(g-Y)\theta_1 \right).
\end{align*}

Recasting the integrands over $C_1$, we get
 \begin{align*}
 \frac{1}{2 \pi i} \int\limits_{C_{1}} \frac{(1-u^{2g-2Y})(1+u) \dif u}{u(1-u) (1-v/u^2)}=& \sum_{j=0}^{2g-2Y-1} \sum_{n=0}^{\infty}v^n \frac{1}{2 \pi i} \int_{C_{1}} u^{j-2n-1}(1+u) \, \dif u \\
 =& \sum_{j=0}^{2g-2Y-1} \sum_{n=0}^{\infty}v^n \int\limits_{-\theta_1}^{\theta_1} ( e^{2 \pi i \theta(j-2n)}+ e^{2 \pi i \theta(j+1-2n)})\, \dif \theta \\
 =& \frac{1}{ \pi } \sum_{j=0}^{2g-2Y-1} \sum_{n=0}^{\infty}v^n \Big (\frac{ \sin(2 \pi (2n-j) \theta_1)}{2n-j}+\frac{ \sin(2 \pi (2n-j-1) \theta_1)}{2n-j-1} \Big ),
 \end{align*}
and similarly,
 \begin{align*}
\frac{1}{2 \pi i} \int\limits_{C_{1}} \frac{uv(1-u^{2g-2X})(1+u) \dif u}{(1-u) (1-u^2v)} %= &  \sum_{j=0}^{2g-2Y-1} \sum_{n=1}^{\infty}v^n \frac{1}{2 \pi i} \int_{C_{1}} u^{j+2n-1}(1+u) \, du \\
% =& \sum_{j=0}^{2g-2Y-1} \sum_{n=1}^{\infty}v^n \int_{-\theta_1}^{\theta_1} ( e^{2 \pi i \theta(j+2n)}+ e^{2 \pi i \theta(j+1+2n)}) \, d \theta \\
=& \frac{1}{ \pi } \sum_{j=0}^{2g-2Y-1} \sum_{n=1}^{\infty}v^n \Big (\frac{ \sin(2 \pi (2n+j) \theta_1)}{2n+j}+\frac{ \sin(2 \pi (2n+j+1) \theta_1)}{2n+j+1} \Big ).
\end{align*}

  We set
\begin{align*}
 S_{+}(n,j;\theta_1)=& \frac{ \sin(2 \pi (2n+j) \theta_1)}{2n+j}+\frac{ \sin(2 \pi (2n+j+1) \theta_1)}{2n+j+1}, \quad \mbox{and} \\
 S_{-}(n,j;\theta_1)=& \frac{ \sin(2 \pi (2n-j) \theta_1)}{2n-j}+\frac{ \sin(2 \pi (2n-j-1) \theta_1)}{2n-j-1}.
\end{align*}
  Thus we have
\begin{align*}
  S_{21}&(Y) \\
=& \frac{1}{\pi} \frac {X}{\log_q X} \frac {d_A(l_1)}{\sqrt{|l_1|}}\frac{1}{2 \pi i} \oint\limits_{|v|=r} \sum_{j=0}^{2g-2Y-1} \left(\sum_{n=0}^{g-\deg(l_1)/2}v^n S_{-}(n,j;\theta_1)+\sum_{n=1}^{g-\deg(l_1)/2}v^n S_{+}(n,j;\theta_1) \right)\frac{C(v/q;l_1)(1-v^2/q) \, \dif v}{v^{g-\deg(l_1)/2+1}(1-v)^3} \\
& \hspace*{2cm} +O \left( \frac {g(g+\deg(l))(g-Y)X\theta_1}{\log_q X}+|l|^{\varepsilon}X^{1/2+\varepsilon}(g-Y)\theta_1 \right).
 \end{align*}

 We dilate the contour by taking $r$ to infinity and compute the residue at $v=1$.  This renders
 \begin{align*}
  S_{21}(Y)  = \frac{1}{2\pi} \frac {X}{\log_q X} \frac {d_A(l_1)}{\sqrt{|l_1|}} & \sum_{j=0}^{2g-2Y-1}
\left( \sum_{n=0}^{g-\deg(l_1)/2}  S_{-}(n,j;\theta_1)P(n)+ \sum_{n=1}^{g-\deg(l_1)/2}  S_{+}(n,j;\theta_1)P(n) \right)\\
&   +O \left( \frac {g(g+\deg(l))(g-Y)X\theta_1}{\log_q X}+|l|^{\varepsilon}X^{1/2+\varepsilon}(g-Y)\theta_1 \right),
 \end{align*}
 where
\begin{align*}
P(x) = C(1/q;l_1) & \Big(1-\frac{1}{q}\Big)(g-\deg(l_1)/2-x+1)(g-\deg(l_1)/2-x+2)\\
&-2(g-\deg(l_1)/2-x+1)\Big( \frac 1q C'(1/q;l_1) \Big(1-\frac 1q \Big)-C(1/q;l_1)\frac 2q \Big )\\
&+\frac {1}{q^2} \Big( 1-\frac 1q \Big) C''(1/q;l_1)-\frac 4{q^2}C'(1/q;l_1).
\end{align*}
 Using \cite[Lemma 9.4]{Florea17-2}, we obtain that
\begin{align*}
%%\label{E1}
 S_{21}(Y) = \frac {X}{\log_q X} \frac {d_A(l_1)}{\sqrt{|l_1|}} & C(1/q;l_1)\frac{(g-\deg(l_1)/2)^2(g-Y)}{2\zeta_q(2)}+ O \left( |l_1|^{\varepsilon} (g-\deg(l_1)/2)(g-Y)(\log g) \frac {X}{\log_q X} \right. \\
&  + |l_1|^{\varepsilon}(g-\deg(l_1)/2)(g-Y)^3\theta_1\frac {X}{\log_q X} + |l_1|^{\varepsilon}(g-\deg(l_1)/2)(g-Y)\theta_1^{-1}\frac {X}{\log_q X}  \\
& \left. + |l_1|^{\varepsilon}(g-Y)\theta_1^{-2}\frac {X}{\log_q X}+ \frac {g(g+\deg(l))(g-Y)X\theta_1}{\log_q X}+|l|^{\varepsilon}X^{1/2+\varepsilon}(g-Y)\theta_1 \right).
\end{align*}

We now set $\theta_1=1/\sqrt{g-\deg(l_1)/2}$.  The above formula, together with \eqref{E2}, implies that
\begin{align}
\label{S2}
\begin{split}
 S_{2}(Y) = \frac {X}{\log_q X} & \frac {d_A(l_1)}{\sqrt{|l_1|}} C(1/q;l_1)\frac{(g-\deg(l_1)/2)^2(g-Y)}{2\zeta_q(2)} \\
  &+O\left( |l|^{\varepsilon}(g-\deg(l_1)/2)^{1/2}(g-Y)^3\frac {X}{\log_q X} +|l|^{\varepsilon}(g-\deg(l_1)/2)^{3/2+\varepsilon}(g-Y)\frac {X}{\log_q X} \right. \\
& \hspace*{1cm}  \left. + g(g+\deg(l))(g-Y)(g-\deg(l_1)/2)^{-1/2}\frac {X}{\log_q X}+|l|^{\varepsilon}X^{1/2+\varepsilon}(g-Y)(g-\deg(l_1)/2)^{-1/2} \right).
\end{split}
\end{align}

 Similarly, replacing $\deg(l_1)$ by $\deg(l_1)+1$ throughout the above computation gives $S_{2}(Y)$ for the case in which $\deg(l_1)$ is odd.

\subsection{Conclusion}
  We now evaluate $S_1(Y)$ which, we recall, is given by
\begin{align*}
 S_1(Y)=\sum_{P\in \mathcal{P}_{2g+1,q}}\sum_{f \in \mathcal{M}_{\leq 2Y}} \frac{d_{A}(f)\chi_P(lf)}{\sqrt{|f|}}+ \sum_{P\in \mathcal{P}_{2g+1,q}}\sum_{f \in \mathcal{M}_{\leq 2Y-1}} \frac{d_{A}(f)\chi_P(lf)}{\sqrt{|f|}}.
\end{align*}

  We further apply Lemma \ref{PVlemma} and argue in a similar way as the proof of Lemma~\ref{lem:Ltwistsbound}.  We are led to
\begin{align}
\label{S1Y}
 S_1(Y)= \frac {X}{\log_q X} \sum_{\substack{f \in \mathcal{M}_{\leq 2Y} \\ fl=\square}} \frac{d_{A}(f)}{\sqrt{|f|}}+ \frac {X}{\log_q X} \sum_{\substack{f \in \mathcal{M}_{\leq 2Y-1} \\ fl=\square}} \frac{d_{A}(f)}{\sqrt{|f|}}+O \left( \frac {g(g+\deg(l))X^{1/2}q^Y}{\log_q X}+|l|^{\varepsilon}q^{Y+\varepsilon} \right).
\end{align}

 Once again writing $f$ in the form $l_1f^{2}$ when $fl=\square$, we apply Perron's formula given in Lemma \ref{lem4} to see that  when $\deg(l_1)$ is even, the above sums over $f$ can be written as
\begin{align*}
 &  \frac {d_A(l_1)}{\sqrt{|l_1|}}\frac{1}{2\pi i}\oint\limits_{|u|=r}\frac{ C(u/q;l_1)\mathcal{Z}(u/q)^3}{\mathcal{Z}(u^2/q^2)}\frac{(1+u) \dif u}{u^{Y-\deg(l_1)/2+1}(1-u)}=  \frac {d_A(l_1)}{\sqrt{|l_1|}}\frac{1}{2\pi i}\oint\limits_{|u|=r}\frac{C(u/q;l_1)(1-u^2/q)(1+u) \dif u}{u^{Y-\deg(l_1)/2+1}(1-u)^4},
\end{align*}
  where $0<r<1$. \newline

 We now enlarge the contour of integration to pass the pole at $u=1$ to see that the above equals to
\begin{align}
\label{S1main}
\begin{split}
 \frac {d_A(l_1)}{\sqrt{|l_1|}} & \left(C(1/q;l_1)\left( \frac{(Y-\deg(l_1)/2)^3}{6\zeta_q(2)}+ \left( Y-\frac{\deg(l_1)}{2} \right)^2 \right)-\frac {1}{q} C'(1/q;l_1)\frac{(Y-\deg(l_1)/2)^2}{2\zeta_q(2)} \right) \\
& +  \frac {d_A(l_1)}{\sqrt{|l_1|}}\left(C(1/q;l_1)\left( \frac{(Y-\deg(l_1)/2-1)^3}{6\zeta_q(2)}+ \left( Y- \frac{\deg(l_1)}{2}-1 \right)^2 \right)-\frac {1}{q} C'(1/q;l_1)\frac{(Y-\deg(l_1)/2-1)^2}{2\zeta_q(2)} \right) \\
& \hspace*{2cm} + O \left( \frac {d_A(l_1)}{\sqrt{|l_1|}}|l_1|^{\varepsilon}(Y-\deg(l_1)/2) \right) .
\end{split}
\end{align}

   Applying \eqref{S1main} in \eqref{S1Y}, we conclude that when $\deg(l_1)$ is even,
\begin{align}
\label{S1Yexp}
\begin{split}
 S_1&(Y) \\
 =& \frac {X}{\log_q X} \frac {d_A(l_1)}{\sqrt{|l_1|}} \left( C(1/q;l_1)\left( \frac{(Y-\deg(l_1)/2)^3}{6\zeta_q(2)}+ \left( Y- \frac{\deg(l_1)}{2} \right)^2 \right)-\frac {1}{q} C'(1/q;l_1)\frac{(Y-\deg(l_1)/2)^2}{2\zeta_q(2)} \right) \\
&+ \frac {X}{\log_q X} \frac {d_A(l_1)}{\sqrt{|l_1|}}\left( C(1/q;l_1)\left( \frac{(Y-\deg(l_1)/2-1)^3}{6\zeta_q(2)}+ \left( Y-\frac{\deg(l_1)}{2}-1 \right)^2 \right)-\frac {1}{q} C'(1/q;l_1)\frac{(Y-\deg(l_1)/2-1)^2}{2\zeta_q(2)} \right) \\
& \hspace*{2cm} + O \left( \frac {d_A(l_1)}{\sqrt{|l_1|}}|l_1|^{\varepsilon}(Y-\deg(l_1)/2))  \frac {X}{\log_q X} + \frac {g(g+\deg(l))X^{1/2}q^Y}{\log_q X}+|l|^{\varepsilon}q^{Y+\varepsilon} \right).
\end{split}
\end{align}

  Notice that we have
\begin{align}
\label{comb1}
\begin{split}
  \frac{(Y-\deg(l_1)/2)^3}{6\zeta_q(2)}+& \left( Y- \frac{\deg(l_1)}{2} \right)^2  +\frac{(g-\deg(l_1)/2)^2(g-Y)}{2\zeta_q(2)}\\
=&\frac{(g-\deg(l_1)/2)^3}{6\zeta_q(2)}+ \left( g-\frac{\deg(l_1)}{2} \right)^2 +O \left( (g-\deg(l_1)/2)(g-Y)^2 \right)
\end{split}
\end{align}
and similarly
\begin{align} \label{comb2}
\begin{split}
 \frac{(Y-\deg(l_1)/2-1)^3}{6\zeta_q(2)}+ & \left( Y- \frac{\deg(l_1)}{2}-1 \right)^2 +\frac{(g-\deg(l_1)/2)^2(g-Y)}{2\zeta_q(2)}\\
=&\frac{(g-\deg(l_1)/2)^3}{6\zeta_q(2)}+\frac 12 \left( 1+\frac 1q \right) \left( g- \frac{\deg(l_1)}{2} \right)^2 +O\left( (g-\deg(l_1)/2)(g-Y)^2 \right).
\end{split}
\end{align}

   We now combine \eqref{S2}, \eqref{S1Yexp} and make use of \eqref{comb1}, \eqref{comb2} to deduce \eqref{eq:1ndmomentE} for the case when $\deg(l_1)$ is even, upon setting $Y=g-[100 \log g]$. The case in which $\deg(l_1)$ is odd can be similarly established. This completes the proof of Proposition \ref{Prop1}.

\section{Proof of Theorem \ref{thmorderofmag}}
\label{sec:upper bd}

\subsection{Setup}
Since the result of the theorem in the case of $k=1$ follows from Propositions~\ref{Prop1'} and~\ref{Prop-twisted1stmoment}, we may assume that $k \neq 1$ in the sequal.  Recall that we have $X=q^{2g+1}$. Replacing $k$ by $2k$ and noting \eqref{Lpostive}, we see that in order to prove Theorem \ref{thmorderofmag}, it suffices to show that for any real $k>0$,
\begin{align}\label{lowerbounds}
\sum_{D\in\mathcal{H}_{2g+1,q}} L(\tfrac{1}{2},\chi_{D})^{2k}\gg_{k}  X(\log_q X)^{k(2k+1)},  \quad
\sum_{P\in\mathcal{P}_{2g+1,q}} L(\tfrac{1}{2},\chi_{P})^{2k}\gg_{k}  X(\log_q X)^{k(2k+1)-1},
\end{align}
  and
\begin{align} \label{upperbounds}
\sum_{D\in\mathcal{H}_{2g+1,q}} L(\tfrac{1}{2},\chi_{D})^{2k}\ll_{k} X(\log_q X)^{k(2k+1)},
\sum_{P\in\mathcal{P}_{2g+1,q}} L(\tfrac{1}{2},\chi_{P})^{2k}\ll_{k} X(\log_q X)^{k(2k+1)-1}.
\end{align}

We follow the approach of A. J. Harper in \cite{Harper} and define for a large number $M$ depending on $k$ only,
\begin{align}
\label{alphadef}
 \alpha_{0} = \frac{\log 2}{\log X}, \;\;\;\;\; \alpha_{j} = \frac{20^{j-1}}{(\log\log X)^{2}} \;\;\; \mbox{for all} \; j \geq 1, \quad
\mathcal{J} = \mathcal{J}_{k, X} = 1 + \max\{j : \alpha_{j} \leq 10^{-M} \} .
\end{align}

Applying the above notations and Lemma \ref{RS}, we see that for $X$ large enough,
\begin{align}
\label{sump1}
 \mathcal{J} \leq \log\log\log X , \;\;\;\;\; \alpha_{1} = \frac{1}{(\log\log X)^{2}} , \;\;\;\;\; \sum_{|P| \leq X^{1/(\log\log X)^{2}}}
\frac{1}{|P|} \leq \log\log X.
\end{align}

  Also, for $1 \leq j \leq \mathcal{J}-1$ and $X$ large enough,
\begin{align}
\label{sumpj}
\mathcal{J}-j \leq \frac{\log(1/\alpha_{j})}{\log 20} , \;\;\;\;\; \sum_{X^{\alpha_{j}} < |P| \leq X^{\alpha_{j+1}}} \frac{1}{|P|}
 = \log \alpha_{j+1} - \log \alpha_{j} + o(1) = \log 20 + o(1) \leq 10.
\end{align}

Combining \eqref{sump1} and \eqref{sumpj}, we arrive at
\begin{align} \label{sumalphaj}
   \sum_{X^{\alpha_{j-1}} < |P| \leq X^{\alpha_{j}}}\frac 1{|P|} \leq \frac{100}{10^{3M/4}}\alpha^{-3/4}_j, \quad 1\leq j  \leq \mathcal{J}.
\end{align}
To see this, note that for $2 \leq j \leq \mathcal{J}-1$, $1 \leq (\alpha_j 10^{M})^{-1}$.  So the bound in \eqref{sumalphaj} follows from the second inequality in \eqref{sumpj}.  Furthermore,
\[ \frac{100}{10^{3M/4}}\alpha^{-3/4}_1 = \frac{100}{10^{3M/4}} ( \log \log X)^{3/2} \geq \log \log X \]
for all sufficient large $X$.  So \eqref{sumalphaj} in the case of $j=1$ follows from the last bound in \eqref{sump1}. \newline

For any real numbers $x, y $ with $y \geq 0$, we set
\begin{align} \label{E}
  E_{y}(x) = \sum_{j=0}^{2\lceil y \rceil} \frac {x^{j}}{j!}.
\end{align}
  We then define for $D\in\mathcal{H}_{2g+1,q}$, any real number $\alpha$ and any $1\leq j  \leq \mathcal{J}$,
\begin{align}
\label{defP}
 {\mathcal P}_j(D)=&  \sum_{\substack{ X^{\alpha_{j-1}} < |Q| \leq X^{\alpha_{j}} \\ Q \text{ prime} }}  \frac{\chi_{D} (Q)}{\sqrt{|Q|}}, \quad {\mathcal M}_j(D, \alpha) = E_{e^5\alpha^{-3/4}_j} \Big (\alpha {\mathcal P}_j(D) \Big ) \quad \mbox{and} \quad {\mathcal M}(D, \alpha)=  \prod^{\mathcal{J}}_{j=1} {\mathcal M}_j(D, \alpha).
\end{align}

  Note that each ${\mathcal M}_j(D,\alpha)$ is a short Dirichlet polynomial of length at most $X^{2\alpha_{j}\lceil e^5\alpha^{-3/4}_j \rceil}$. By taking $X$ large enough, we have that
\begin{align*}
%%\label{exponentbound}
 \sum^{\mathcal{J}}_{j=1} 2\alpha_{j}\lceil e^5\alpha^{-3/4}_j \rceil \leq 5e^510^{-M/4}.
\end{align*}
   It follows that ${\mathcal M}(D, \alpha)$ is also a short Dirichlet polynomial of length not exceeding $X^{5e^510^{-M/4}}$. \newline

  As ${\mathcal P}_j(D)$ is real, it follows from \cite[Lemma 1]{Radziwill&Sound} that ${\mathcal M}_j(D, \alpha), 1 \leq j \leq \mathcal{J}$ and  ${\mathcal M}(D, \alpha)$ are all positive.  If $2k>1$, H\"older's inequality yields
\begin{align}
\label{basicbound1}
\begin{split}
 \sum_{D\in\mathcal{H}_{2g+1,q}} & L(\tfrac{1}{2},\chi_{D}) \mathcal{M}(D, 2k-1) \\
 & \leq \Big ( \sum_{D\in\mathcal{H}_{2g+1,q}}|L(\tfrac{1}{2},\chi_{D})|^{2k} \Big )^{1/(2k)}\Big ( \sum_{D\in\mathcal{H}_{2g+1,q}}
\mathcal{M}(D, 2k-1)^{2k/(2k-1)} \Big)^{(2k-1)/(2k)}.
\end{split}
\end{align}

If $0<2k<1$, we proceed as in \cite[Section 2]{Gao2021-3}.  This gives that for any real number $\alpha$,
\begin{align*}
%%\label{Nprodbound}
 \mathcal{M}(D, \alpha)\mathcal{M}(D, -\alpha)  \geq 1.
\end{align*}
   We set  $\alpha=(2/k-3)^{-1}$ above to note that $0 < \alpha < 1$ when $0 < 2k < 1$. This implies that
\begin{align*}
%%\label{basicbound}
\begin{split}
 \sum_{D\in\mathcal{H}_{2g+1,q}} & L(\tfrac{1}{2},\chi_{D})  \mathcal{M}(D, 2k-1)  \\
 \leq & \sum_{D\in\mathcal{H}_{2g+1,q}}L(\tfrac{1}{2},\chi_{D})^{c}\cdot L(\tfrac{1}{2},\chi_{D})^{1-c} \mathcal{M}(D, 2k-2)^{(1-c)/2}  \cdot
 \mathcal{M}(D,2k-1) \mathcal{M}(D, 2-2k)^{(1-c)/2}.
\end{split}
\end{align*}

  We then apply H\"older's inequality with exponents $2k/c$, $2/(1-c)$, $((1+c)/2-c/(2k))^{-1}$ to the
  last sum above, by noting that these exponents are all greater than $1$ upon taking $0< c < 2k$. This leads to
\begin{align}
\label{basicboundksmall}
\begin{split}
 \sum_{D\in\mathcal{H}_{2g+1,q}} & L(\tfrac{1}{2},\chi_{D}) \mathcal{M}(D, 2k-1) \\
&  \leq  \Big ( \sum_{D\in\mathcal{H}_{2g+1,q}}L(\tfrac{1}{2},\chi_{D})^{2k} \Big )^{c/(2k)}\Big (
\sum_{D\in\mathcal{H}_{2g+1,q}}L(\tfrac{1}{2},\chi_{D})^2\mathcal{M}(D, 2k-2) \Big)^{(1-c)/2} \\
 & \hspace*{2cm}  \times \Big (\sum_{D\in\mathcal{H}_{2g+1,q}}\mathcal{M}(D,2k-1)^{2(2-3k)/(1-2k)}\mathcal{M}(D, 2-2k)^{2}
 \Big)^{(1+c)/2-c/(2k)}.
\end{split}
\end{align}

  Similar arguments apply to sums involving with $L(\tfrac{1}{2},\chi_{P})$ for $P \in \mathcal{P}_{2g+1,q}$. We then
deduce from \eqref{basicbound1} and \eqref{basicboundksmall} that in order to prove \eqref{lowerbounds}, it suffices to establish the following propositions.
\begin{proposition}
\label{Prop4} With notations as above, we have for $k > 0$,
\begin{align*}
%%\label{LprimeN}
\sum_{D\in\mathcal{H}_{2g+1,q}}L(\tfrac{1}{2},\chi_{D}) {\mathcal M}(D, 2k-1) \gg X(\log_q X)^{((2k)^2+1)/2} \; \mbox{and} \;
\sum_{P\in\mathcal{P}_{2g+1,q}}L(\tfrac{1}{2},\chi_{P}) {\mathcal M}(P, 2k-1) \gg  X(\log_q X)^{ ((2k)^2-1)/2}.
\end{align*}
\end{proposition}

\begin{proposition}
\label{Prop5} With notations as above, we have  $k > 1/2$,
\begin{align*}
%%\label{Nestmation}
\sum_{D\in\mathcal{H}_{2g+1,q}}\mathcal{M}(D, 2k-1)^{2k/(2k-1)}  \ll X ( \log_q X  )^{(2k)^2/2} \; \mbox{and} \;
\sum_{P\in\mathcal{P}_{2g+1,q}}\mathcal{M}(P, 2k-1)^{2k/(2k-1)}  \ll X ( \log_q X  )^{(2k)^2/2-1}.
\end{align*}
  Also, we have for $0<k<1/2$, we have
\begin{align*}
%%\label{Nestmation}
\sum_{D\in\mathcal{H}_{2g+1,q}}\mathcal{M}(D,2k-1)^{\frac {2(2-3k)}{1-2k}}\mathcal{M}(D, 2-2k)^{2} \ll & X ( \log_q X  )^{(2k)^2/2} \quad \mbox{and} \\
 \sum_{P\in\mathcal{P}_{2g+1,q}}\mathcal{M}(P,2k-1)^{\frac {2(2-3k)}{1-2k}}\mathcal{M}(P, 2-2k)^{2} \ll & X ( \log_q X  )^{(2k)^2/2-1}.
\end{align*}
\end{proposition}

\begin{proposition}
\label{Prop6}
  With notations as above, we have for $0<k<1/2$,
\begin{align}
\label{L2estmation}
\begin{split}
\sum_{D\in\mathcal{H}_{2g+1,q}}L(\tfrac{1}{2},\chi_{D})^2{\mathcal M}(D, 2k-2)  \ll X ( \log_q X  )^{((2k)^2+2)/2} \; \mbox{and} \; \sum_{P\in\mathcal{P}_{2g+1,q}}L(\tfrac{1}{2},\chi_{P})^2{\mathcal M}(P, 2k-2)  \ll X ( \log_q X  )^{2k^2}.
\end{split}
\end{align}
\end{proposition}

Indeed, assuming the truth of Propositions~\ref{Prop4}, \ref{Prop5} and \ref{Prop6}, it emerges from \eqref{basicboundksmall} that
\[ \Big( \sum_{D\in\mathcal{H}_{2g+1,q}}L(\tfrac{1}{2},\chi_{D})^{2k} \Big)^{c/(2k)} \gg X^{c/(2k)} (\log_q X)^{(2k+1)c/2} . \]
Now the lower bound for the $D$-sum in \eqref{lowerbounds} follows easily from this and the minorant for the $P$-sum therein can be worked out similarly. \newline

  Our proofs of the above propositions are similar to those of Propositions 3.4--3.6 in \cite{G&Zhao11}. Also, the proof of \eqref{upperbounds} is similar to that of Theorem 1.2 in \cite{G&Zhao11}.  We shall therefore only include the proofs of Proposition \ref{Prop6} and \eqref{upperbounds} in the remainder of the paper. We point out that in the course of establishing Proposition \ref{Prop4}, we need to make use of Proposition \ref{Prop1'} and Proposition \ref{Prop-twisted1stmoment}.

\subsection{Proof of Proposition \ref{Prop6} }
\label{sec 4}

  Due to the similarities in their proofs, we only present that of the second bound in \eqref{L2estmation} here. For $D \in \mathcal{H}_{2g+1,q}$, we define
\begin{align}
\label{defM}
{\mathcal M}_{i,j}(D) =& \sum_{\substack{X^{\alpha_{i-1}} < |Q| \leq X^{\alpha_{i}} \\ Q \text{ prime} }}  \frac{\chi_{D} (Q)}{|Q|^{1/2+1/(\log X^{\alpha_{j}})}}
\frac{\log (X^{\alpha_{j}}/|Q|)}{\log X^{\alpha_{j}}}, \quad 1\leq i \leq j \leq \mathcal{J} .
\end{align}
  We also define for $0 \leq j \leq \mathcal{J}-1$,
\begin{align*}
 \mathcal{S}(j) = & \left\{ P \in \mathcal{P}_{2g+1,q}: | {\mathcal M}_{i,l}(P)|  \leq \alpha_{i}^{-3/4} \; \mbox{for all} \; 1 \leq i \leq j \; \mbox{and} \; i
\leq l \leq \mathcal{J}, \right. \\
& \hspace*{4cm}  \left. \mbox{but} \; | {\mathcal M}_{j+1,l}(P)| > \alpha_{j+1}^{-3/4} \; \mbox{for some} \; j+1 \leq l \leq \mathcal{J} \right\} \quad \mbox{and} \\
 \mathcal{S}(\mathcal{J}) =& \left\{ P \in \mathcal{P}_{2g+1,q}: |{\mathcal M}_{i, \mathcal{J}}(P)| \leq \alpha_{i}^{-3/4} \; \mbox{for all} \; 1 \leq i \leq \mathcal{J} \right\}.
\end{align*}

   Note that
$$ \mathcal{P}_{2g+1,q}= \bigcup_{j=0}^{ \mathcal{J}} \mathcal{S}(j).  $$
Thus, in order to establish the second bound in \eqref{L2estmation}, it suffices to show that
\begin{align}
\label{LM}
   \sum_{j=0}^{\mathcal{J}}\sum_{P \in \mathcal{S}(j)}L(\tfrac{1}{2}, \chi_{P})^2 \mathcal{M}(P, 2k-2)
   \ll X ( \log_q X  )^{2k^2}.
\end{align}

  Observe that
\begin{align}
\label{S0est}
\begin{split}
\text{meas}(\mathcal{S}(0)) \ll & \sum_{P \in \mathcal{P}_{2g+1,q}}
\sum^{\mathcal{J}}_{l=1}
\Big ( \alpha^{3/4}_{1}{|\mathcal
M}_{1, l}(P)| \Big)^{2\lceil 1/(10\alpha_{1})\rceil } .
\end{split}
\end{align}

    We write for $1\le l \le \mathcal{J}$,
\begin{align}
\label{MDirichPoly}
 {\mathcal M}_{1, l}(P)^{2\lceil 1/(10\alpha_{1})\rceil } =&  (2\lceil 1/(10\alpha_{1})\rceil)! \sum_{ \substack{ \Omega(f) = 2\lceil 1/(10\alpha_{1})\rceil \\ Q|f \implies  |Q| \leq X^{\alpha_{1}} \\ Q \text{ prime}}}
  \frac{\lambda_l(f)}{\sqrt{|f|}}\frac{\chi_D(f) }{w(f)} .
\end{align}
Here $\Omega(f)$ is the number of prime powers dividing $f$, $w(f)$ the multiplicative function such that
    $w(P^{\alpha}) = \alpha!$ for prime powers $P^{\alpha}$ and $\lambda(f)$ the completely multiplicative function with
\[    \lambda_l(P)= |P|^{-1/(\log X^{\alpha_{l}})} \frac{\log (X^{\alpha_{l}}/|Q|)}{\log X^{\alpha_{l}}} \]
on primes $P$. \newline

  We substitute \eqref{MDirichPoly} into the last expression in \eqref{S0est} and then apply \eqref{wsum} to estimate $\sum_{P \in \mathcal{P}_{2g+1,q}}{\mathcal M}_{1, l}(P)^{2\lceil 1/(10\alpha_{1})\rceil }$ for a fixed $l$. Note that the Dirichlet polynomial given \eqref{MDirichPoly} is a short Dirichlet polynomial whose coefficients are all $\ll X^{\varepsilon}$. It follows that we may ignore the contribution from the error term in \eqref{MDirichPoly} (arising from the $O$-term in \eqref{wsum}) and focus on the contribution from the main term.  This leads to
\begin{align}
\label{maintermbound}
\begin{split}
 \sum_{P \in \mathcal{P}_{2g+1,q}}{\mathcal M}_{1, l}(P)^{2\lceil 1/(10\alpha_{1})\rceil } \ll & \Big (\frac {X}{\log_q X}\Big )\Big( \alpha^{3/4}_{1} \Big)^{2\lceil 1/(10\alpha_{1})\rceil}(2\lceil 1/(10\alpha_{1})\rceil)!\sum_{ \substack{ f=\square \\ \Omega(f) = 2\lceil 1/(10\alpha_{1})\rceil \\ Q|f \implies  |Q| \leq X^{\alpha_{l}} \\ Q \text{ prime}}} \frac{\lambda_l(f)}{\sqrt{|f|}}\frac{1 }{w(f)}.
\end{split}
\end{align}
   
    Upon replacing $f$ by $f^2$ in the sum above and noting that $\lambda_l(f^2)=\lambda^2_l(f) \leq 1$ and $1/w(f^2) \leq 1/w(f)$, we deduce from \eqref{maintermbound} that
\begin{align}
\label{S0upperbound}
\begin{split}
 \sum_{P \in \mathcal{P}_{2g+1,q}}{\mathcal M}_{1, l}(P)^{2\lceil 1/(10\alpha_{1})\rceil }  \ll & \Big (\frac {X}{\log_q X} \Big )\Big( \alpha^{3/4}_{1} \Big)^{2\lceil 1/(10\alpha_{1})\rceil}\frac {(2\lceil 1/(10\alpha_{1})\rceil)!}{(\lceil 1/(10\alpha_{1})\rceil)!}\Big (\sum_{|P| \leq X^{\alpha_{1}}}\frac{\lambda^2_l(P)}{|P|} \Big )^{\lceil 1/(10\alpha_{1})\rceil} \\
 \ll & \Big (\frac {X}{\log_q X} \Big )\Big( \alpha^{3/4}_{1} \Big)^{2\lceil 1/(10\alpha_{1})\rceil}\frac {(2\lceil 1/(10\alpha_{1})\rceil)!}{(\lceil 1/(10\alpha_{1})\rceil)!}\Big (\sum_{|P| \leq X^{\alpha_{1}}}\frac{1}{|P|} \Big )^{\lceil 1/(10\alpha_{1})\rceil} \\
 \ll & \Big (\frac {X}{\log_q X} \Big )\Big( \alpha^{3/4}_{1} \Big)^{2\lceil 1/(10\alpha_{1})\rceil}\frac {(2\lceil 1/(10\alpha_{1})\rceil)!}{(\lceil 1/(10\alpha_{1})\rceil)!}(\log \log X  )^{\lceil 1/(10\alpha_{1})\rceil},
\end{split}
\end{align}
where the last estimation above follows from \eqref{lam2p}. \newline
   
   Further applying the estimates
\begin{align*} 
%\label{Stirling}
  \Big( \frac ne \Big)^n \leq n! \leq n \Big( \frac ne \Big)^n,
\end{align*}
  we deduce that
\begin{align*}
 & \Big (\frac {X}{\log_q X} \Big )\Big( \alpha^{3/4}_{1} \Big)^{2\lceil 1/(10\alpha_{1})\rceil}\frac {(2\lceil 1/(10\alpha_{1})\rceil)!}{(\lceil 1/(10\alpha_{1})\rceil)!}
\leq  2\lceil 1/(10\alpha_{1})\rceil\Big (\frac {X}{\log_q X} \Big ) \Big( \frac{4 \lceil 1/(10\alpha_{1})\rceil  \alpha^{3/2}_{1}}{e} \Big)^{\lceil 1/(10\alpha_{1})\rceil}.
\end{align*}
   
   We apply the above and the observation that $\log \log X=\alpha^{-1/2}_1$ in the last expression in \eqref{S0upperbound} to arrive at 
\begin{align*}
\sum_{P \in \mathcal{P}_{2g+1,q}}{\mathcal M}_{1, l}(P)^{2\lceil 1/(10\alpha_{1})\rceil }\ll 
X e^{-1/(12\alpha_{1})}.
\end{align*}   
   
   We now sum over $l$ to conclude that
\begin{align}
\label{S0bound}
\text{meas}(\mathcal{S}(0)) \ll &
\mathcal{J}X e^{-1/(12\alpha_{1})}\ll X e^{-(\log\log X)^{2}/20}  .
\end{align}

   We then deduce via H\"older's inequality that
\begin{align}
\label{LS0bound}
\begin{split}
 \sum_{P \in \mathcal{S}(0)}L(\tfrac{1}{2}, \chi_{P})^2 & \mathcal{M}(P, 2k-2) \\
\leq & \Big ( \text{meas}(\mathcal{S}(0)) \Big )^{1/4} \Big (
\sum_{P\in\mathcal{P}_{2g+1,q}}L(\half, \chi_{P})^{8} \Big )^{1/4} \Big ( \sum_{P\in\mathcal{P}_{2g+1,q}} \mathcal{M}(P, 2k-2)^{2} \Big )^{1/2}.
\end{split}
\end{align}

 We apply \eqref{wsum} again and use arguments similar to those in the proof of \cite[Proposition 2.2]{G&Zhao8}.  This gives
\begin{align}
\label{N2k2bound}
 \sum_{P\in\mathcal{P}_{2g+1,q}}\mathcal{M}(P, 2k-2)^{2} \ll X ( \log_q X  )^{O(1)}.
\end{align}

 Now, setting $n=4$ and $\varepsilon=1$ in \eqref{upperboundP} implies that
\begin{align}
\label{L8bound}
\sum_{P\in\mathcal{P}_{2g+1,q}} L(\half, \chi_{P})^{8} \ll X ( \log_q X  )^{O(1)}.
\end{align}
Note that because of the factor $\exp ( - (\log \log X)^2/20)$ in \eqref{S0bound}, \eqref{LM} holds so long as the exponents of the logarithms in \eqref{N2k2bound} and \eqref{L8bound} are bounded above by some absolute constant. \newline

  We apply the estimates in \eqref{S0bound}, \eqref{N2k2bound} and \eqref{L8bound} in \eqref{LS0bound} to deduce that
\begin{align*}
\sum_{P \in \mathcal{S}(0)}L(\tfrac{1}{2}, \chi_{P})^2 \mathcal{M}(P, 2k-2)    \ll X ( \log_q X  )^{2k^2}.
\end{align*}

  Thus we may further assume that $j \geq 1$. When $P \in \mathcal{S}(j)$, we set $x=X^{\alpha_j}$ in Lemma~\ref{prop-logup}.  This leads to
\begin{align*}
\begin{split}
 & L(\tfrac{1}{2}, \chi_{P})^2 \ll (\log_q X) \exp \Big(\frac {2}{\alpha_j} \Big) \exp \Big (2
\sum^j_{i=0}{\mathcal M}_{i,j}(P)\Big ).
\end{split}
 \end{align*}

   As  $|M_{i, j}(P)| \leq  \alpha^{-3/4}_i$ if $P \in \mathcal{S}(j)$, we can argue as in the proof of \cite[Lemma 5.2]{Kirila} by setting $x=m M_{i, j}(P)$, $N=e^5\alpha^{-3/4}_i$ in \cite[(5.2)]{Kirila} to see that, for $|m | \leq 2$,
\begin{align*}
\begin{split}
\exp\Big (m \sum^j_{i=0}{\mathcal M}_{i,j}(P)\Big ) \ll
\prod^j_{i=1} E_{e^5\alpha^{-3/4}_i}(m {\mathcal M}_{i,j}(P)),
\end{split}
 \end{align*}
 with $E_{e^5\alpha^{-3/4}_i}$ is defined in \eqref{E}.  In particular, it follows that
\begin{align*}
\begin{split}
\exp\Big (2
\sum^j_{i=0}{\mathcal M}_{i,j}(P)\Big ) \ll
\prod^j_{i=1}E_{e^5\alpha^{-3/4}_i}(2{\mathcal M}_{i,j}(P)).
\end{split}
 \end{align*}

    We then deduce from the description on $\mathcal{S}(j)$ that when $j \geq 1$,
\begin{align}
\label{LboundinSP}
\begin{split}
\sum_{P \in \mathcal{S}(j)} & L(\tfrac{1}{2}, \chi_{P})^2 \mathcal{M}(P, 2k-2)   \\
 \ll &  (\log_q X) \exp \Big(\frac {2}{\alpha_j} \Big)
 \sum^{ \mathcal{I}}_{l=j+1} \sum_{P\in\mathcal{P}_{2g+1,q}}\exp \Big ( 2 \sum^j_{i=1}{\mathcal M}_{i,j}(P)\Big )\mathcal{M}(P, 2k-2) \Big ( \alpha^{3/4}_{j+1}{ | \mathcal
M}_{j+1, l}(P)  |^{2\lceil 1/(10\alpha_{j+1})\rceil } \Big ) \\
\ll &  \exp \Big(\frac {2}{\alpha_j} \Big)  \sum^{ \mathcal{I}}_{l=j+1}
\sum_{P\in\mathcal{P}_{2g+1,q}}
\prod^j_{i=1}\Big ( E_{e^5\alpha^{-3/4}_i}(2{\mathcal M}_{i,j}(P))E_{e^5\alpha^{-3/4}_i}((2k-2){\mathcal P}_{i}(P))\Big ) \\
& \hspace*{3cm} \times E_{e^5\alpha^{-3/4}_{j+1}}(2{\mathcal P}_{j+1}(P)) \Big ( \alpha^{3/4}_{j+1}{\mathcal
M}_{j+1, l}(P)\Big)^{2\lceil 1/(10\alpha_{j+1})\rceil } \prod^{\mathcal{I}}_{i=j+2}  E_{e^5\alpha^{-3/4}_i}((2k-2){\mathcal P}_{i}(P)).
\end{split}
\end{align}
  We now want to apply \eqref{wsum} to treat the sum over $P$ in the last expression above. This is similar to the proof of \cite[Proposition 2.2]{G&Zhao8}, except that we need to first separate the product 
\begin{align*}
%%\label{LboundinSP}
\begin{split}
 E_{e^5\alpha^{-3/4}_{j+1}}(2{\mathcal P}_{j+1}(P)) \Big ( \alpha^{3/4}_{j+1}{\mathcal
M}_{j+1, l}(P)\Big)^{2\lceil 1/(10\alpha_{j+1})\rceil }.
\end{split}
\end{align*}  

  To do so, note that if $|{\mathcal P}_{j+1}(P)| \leq \alpha^{-3/4}_{j+1}$, we apply \cite[(5.2)]{Kirila} again to obtain
\begin{align}
\label{E1est}
\begin{split}
 E_{e^5\alpha^{-3/4}_{j+1}}(2{\mathcal P}_{j+1}(P)) \ll \exp\Big (2{\mathcal P}_{j+1}(P)\Big ) \ll \exp\Big ( 2\alpha^{-3/4}_{j+1} \Big )
\ll 2^{2\lceil 1/(10\alpha_{j+1})\rceil }.
\end{split}
 \end{align}  
  
   On the other hand, if $|{\mathcal P}_{j+1}(P)| > \alpha^{-3/4}_{j+1}$, then
\begin{align}
\label{E2est}
\begin{split}
   E_{e^5\alpha^{-3/4}_{j+1}}(2{\mathcal P}_{j+1}(P)) \le & \sum_{r=0}^{2\lceil e^5\alpha^{-3/4}_{j+1} \rceil} \frac{|2{\mathcal P}_{j+1}(P)|^r}{r!}  \le
|2{\mathcal P}_{j+1}(P)|^{2\lceil e^5\alpha^{-3/4}_{j+1} \rceil}  \sum_{r=0}^{2\lceil e^5\alpha^{-3/4}_{j+1} \rceil} \Big( \alpha^{3/4}_{j+1} \Big)^{2\lceil e^5\alpha^{-3/4}_{j+1} \rceil-r} \frac{1}{r!}  \\
\le &
|2\alpha^{3/4}_{j+1}{\mathcal P}_{j+1}(P)|^{2\lceil e^5\alpha^{-3/4}_{j+1} \rceil}e^{\alpha^{-3/4}_{j+1}} \leq   \Big(2e\alpha^{3/4}_{j+1}{\mathcal P}_{j+1}(P)\Big )^{2\lceil e^5\alpha^{-3/4}_{j+1} \rceil}.
\end{split}
\end{align}

Using \eqref{E1est}, \eqref{E2est} and Cauchy's inequality, we see that
\begin{align*}
%%\label{LboundinSP}
\begin{split}
E_{e^5\alpha^{-3/4}_{j+1}} & (2{\mathcal P}_{j+1}(P)) \Big ( \alpha^{3/4}_{j+1}{\mathcal
M}_{j+1, l}(P)\Big)^{2\lceil 1/(10\alpha_{j+1})\rceil } \\
\ll & \Big ( 2\alpha^{3/4}_{j+1}{\mathcal
M}_{j+1, l}(P)\Big)^{2\lceil 1/(10\alpha_{j+1})\rceil }+ \Big(2e\alpha^{3/4}_{j+1}{\mathcal P}_{j+1}(P)\Big )^{2\lceil e^5\alpha^{-3/4}_{j+1} \rceil}\Big ( \alpha^{3/4}_{j+1}{\mathcal
M}_{j+1, l}(P)\Big)^{2\lceil 1/(10\alpha_{j+1})\rceil } \\
\ll & \Big ( 2\alpha^{3/4}_{j+1}{\mathcal
M}_{j+1, l}(P)\Big)^{2\lceil 1/(10\alpha_{j+1})\rceil }+ \Big(2e\alpha^{3/4}_{j+1}{\mathcal P}_{j+1}(P)\Big )^{4\lceil e^5\alpha^{-3/4}_{j+1} \rceil}+\Big ( \alpha^{3/4}_{j+1}{\mathcal
M}_{j+1, l}(P)\Big)^{4\lceil 1/(10\alpha_{j+1})\rceil }.
\end{split}
\end{align*}    
  
 We apply the above estimation in \eqref{LboundinSP}.  Due to the similarities of the resulting terms, it suffices to show that 
\begin{align*}
%%\label{LboundinSP}
\begin{split}
\sum_{P \in \mathcal{S}(j)}  L(\tfrac{1}{2}, \chi_{P})^2 & \mathcal{M}(P, 2k-2)   \\
 \ll &  \exp \Big(\frac {2}{\alpha_j} \Big)  \sum^{ \mathcal{I}}_{l=j+1}
\sum_{P\in\mathcal{P}_{2g+1,q}}
\prod^j_{i=1}\Big ( E_{e^5\alpha^{-3/4}_i}(2{\mathcal M}_{i,j}(P))E_{e^5\alpha^{-3/4}_i}((2k-2){\mathcal P}_{i}(P))\Big ) \\
& \hspace*{3cm} \times \Big ( 2\alpha^{3/4}_{j+1}{\mathcal
M}_{j+1, l}(P)\Big)^{2\lceil 1/(10\alpha_{j+1})\rceil } \prod^{\mathcal{I}}_{i=j+2}  E_{e^5\alpha^{-3/4}_i}((2k-2){\mathcal P}_{i}(P)).
\end{split}
\end{align*} 
 
 We are now in the position to apply \eqref{wsum} to estimation the sum over $P$ in the last expression above. Similar to our arguments that lead to the bound of $\text{meas}(\mathcal{S}(0))$ in \eqref{S0bound}, we may consider on the main term contribution from  \eqref{wsum} only, which amounts to extracting perfect squares from the product
\begin{align}
\label{factor}
\begin{split}
\prod^j_{i=1}\Big ( E_{e^5\alpha^{-3/4}_i}(2{\mathcal M}_{i,j}(P))E_{e^5\alpha^{-3/4}_i}((2k-2){\mathcal P}_{i}(P))\Big )  \Big ( 2\alpha^{3/4}_{j+1}{\mathcal
M}_{j+1, l}(P)\Big)^{2\lceil 1/(10\alpha_{j+1})\rceil } \prod^{\mathcal{I}}_{i=j+2}  E_{e^5\alpha^{-3/4}_i}((2k-2){\mathcal P}_{i}(P)).
\end{split}
\end{align}  
 As each factor in the product above involves with distinct primes, it follows that the perfect square terms in the above expression arise from the product of perfect squares in each factor. Using the arguments leading \eqref{S0bound} again together with \eqref{sumpj} , we conclude that the contribution from perfect square terms from the factor
 \[ \Big ( 2\alpha^{3/4}_{j+1}{\mathcal M}_{j+1, l}(P)\Big)^{2\lceil 1/(10\alpha_{j+1})\rceil }\]
 is $\ll e^{-60/\alpha_{j+1}}$.  We then proceed as in the proof of \cite[Proposition 2.2]{G&Zhao8} to see that the total contribution from the square terms of other factors in \eqref{factor} is
\begin{align*}
%%\label{LboundinSP1}
\begin{split}
  \ll \prod_{|P| \leq X^{\alpha_{\mathcal{J}}}}\Big  (1+\frac {(2k)^2}{|P|}+O\Big( \frac 1{|P|^2} \Big) \Big ).
\end{split}
\end{align*} 
 Noting that $20/\alpha_{j+1}=1/\alpha_j$, we then deduce that
\begin{align*}
%%\label{LboundinSP1}
\begin{split}
  \sum_{P \in \mathcal{S}(j)}L(\tfrac{1}{2}, \chi_{P})^2 \mathcal{M}(P, 2k-2)
 \ll  X( \log_q X  )^{2k^2} \exp \Big(\frac {2}{\alpha_j} \Big) (\mathcal{I}-j)e^{-60/\alpha_{j+1}}\ll e^{-1/(10\alpha_{j})}X( \log_q X  )^{2k^2}.
\end{split}
\end{align*}

Summing the right-hand side of the above over $j$, a convergent sum, we derive the bound in \eqref{LM} and complete the proof of Proposition \ref{Prop6}.

\subsection{Sharp upper bounds}

   In this section, we prove the bounds given in \eqref{upperbounds}. Again due to similarities in the proofs, we only establish the first estimate in \eqref{upperbounds} here.  Our arguments are similar to those used in the previous section, so we only give a brief description. \newline

  We apply Lemma \ref{prop-logup} to see that for any $k>0$,
\begin{align} \label{basicest}
L(\tfrac{1}{2},\chi_D) ^{2k} & \leq   \exp \left(2k \sum_{\substack{ |P| \leq x}} \frac{ \chi_{D}(P) }{|P|^{1/2+1/\log x } } \frac {\log (x/|P|)}{\log x}+ k \log \log x+2k\frac {\log X}{\log x}+O(1)\right).
\end{align}

 We recall the definition of ${\mathcal M}_{i,j}(D)$ given in \eqref{defM} for $1\leq i \leq j \leq \mathcal{J}$. We also define for $0 \leq j \leq \mathcal{J}$,
\begin{align*}
 \mathcal{T}(j) =& \left\{ D \in \mathcal{H}_{2g+1,q}: | {\mathcal M}_{i,l}(D)| \leq \alpha_{i}^{-3/4} \; \mbox{for all} \; 1 \leq i \leq j, \; \mbox{and} \; i
\leq l \leq \mathcal{J}, \right. \\
& \hspace*{3cm} \left. \mbox{but} \; | {\mathcal M}_{j+1,l}(D)| > \alpha_{j+1}^{-3/4} \; \text{ for some } j+1 \leq l \leq \mathcal{J} \right\} \quad \mbox{and} \\
 \mathcal{T}(\mathcal{J}) =& \left\{ D \in \mathcal{H}_{2g+1, q}: |{\mathcal M}_{i,
\mathcal{J}}(D)| \leq \alpha_{i}^{-3/4} \; \mbox{for all} \; 1 \leq i \leq \mathcal{J}\right\}.
\end{align*}

   It follows that
$$ \mathcal{H}_{2g+1, q}= \bigcup_{j=0}^{ \mathcal{J}} \mathcal{T}(j).  $$
  Thus, to establish the first bound in \eqref{upperbounds}, it suffices to show that
\begin{align}
\label{sumovermj}
   \sum_{j=0}^{\mathcal{J}}\sum_{D \in \mathcal{T}(j)} L(\half, \chi_D)^{2k}
   \ll X (\log_q X)^{k(2k+1)}.
\end{align}

Similar to the arguments in Section \ref{sec 4},  using the case $n=2k$ and $\varepsilon=1$ in \eqref{upperboundD}, we deduce that
\begin{align*}
\sum_{D \in  \mathcal{T}(0)} L(\half, \chi_D)^{2k} \ll X(\log_q X)^{k(2k+1)}.
\end{align*}

If $D \in \mathcal{T}(j)$ for $j \geq 1$, we set $x=X^{\alpha_j}$ in \eqref{basicest}.  This gives that
\begin{align*}
\begin{split}
 & L(\half, \chi_D)^{2k} \ll (\log_q X)^k \exp \Big(\frac {2k}{\alpha_j} \Big) \exp \Big (2k
\sum^j_{i=0}{\mathcal M}_{i,j}(D)\Big ).
\end{split}
 \end{align*}

   As  $M_{i, j} \leq  \alpha^{-3/4}_i$ when $D \in \mathcal{T}(j)$, we can apply \cite[Lemma 5.2]{Kirila} which yields
\begin{align*}
\begin{split}
\exp\Big (2k
\sum^j_{i=0}{\mathcal M}_{i,j}(D)\Big ) \ll
\prod^j_{i=1}E_{e^2k\alpha^{-3/4}_i}(2k{\mathcal M}_{i,j}(\chi)),
\end{split}
 \end{align*}
  where $E_{e^2k\alpha^{-3/4}_i}$ is defined as in \eqref{E}. \newline

  It follows from the description on $\mathcal{T}(j)$ that if $j \geq 1$,
\begin{align*}
%%\label{LboundinSP}
\begin{split}
  \sum_{D \in \mathcal{T}(j)} L(\half, \chi_D)^{2k} \ll & (\log_q X)^{k} \exp \Big(\frac {2k}{\alpha_j} \Big )
 \sum^{\mathcal{J}}_{l=j+1}  \sum_{\substack{D \in \mathcal{H}_{2g+1,q}}}\exp \Big ( 2k \sum^j_{i=1}{\mathcal M}_{i,j}(D)\Big )\Big (
\alpha^{3/4}_{j+1}{\mathcal
M}_{j+1, l}(D)\Big)^{2k \lceil 1/(10\alpha_{j+1})\rceil } \\
\ll & (\log_q X)^{k} \exp \Big(\frac {2k}{\alpha_j} \Big ) \sum^{\mathcal{J}}_{l=j+1}
\sum_{\substack{D \in \mathcal{H}_{2g+1,q}}}
\prod^j_{i=1} E_{e^2k\alpha^{-3/4}_i}(2k{\mathcal M}_{i,j}(D))\Big ( \alpha^{3/4}_{j+1}{\mathcal
M}_{j+1, l}(D)\Big)^{2k\lceil 1/(10\alpha_{j+1})\rceil } .
\end{split}
 \end{align*}

  We then argue as in the proof of \cite[Theorem 1.3]{G&Zhao10}.   By taking $M$ large enough (recall that $M$ is defined in \eqref{alphadef}),
\begin{align*}
\begin{split}
\sum^{\mathcal{J}}_{l=j+1} \sum_{\substack{D \in \mathcal{H}_{2g+1,q}}} &
\prod^j_{i=1}E_{e^2k\alpha^{-3/4}_i}(2k{\mathcal M}_{i,j}(D))\Big ( \alpha^{3/4}_{j+1}{\mathcal
M}_{j+1, l}(D)\Big)^{2k \lceil 1/(10\alpha_{j+1})\rceil } \\
\ll & X(\mathcal{J}-j)e^{-42k/\alpha_{j+1}}  \prod_{|P| \leq X^{\alpha_j}}\Big  (1+\frac {(2k)^2}{2|P|}+O\Big( \frac 1{|P|^2} \Big) \Big ) \ll  e^{-41k/\alpha_{j+1}} X (\log_q X)^{2k^{2}}.
\end{split}
 \end{align*}

 This implies that
\begin{align*}
\begin{split}
 & \sum_{D \in \mathcal{T}(j)} L(\half, \chi_D)^{2k}
\ll   e^{-k/(20\alpha_{j})} X (\log_q X)^{k(2k+1)}.
\end{split}
 \end{align*}

Summing over the convergent sum over $j$, the above implies that \eqref{sumovermj} holds so that the estimates in \eqref{upperbounds} are valid and this completes the proof of Theorem \ref{thmorderofmag}.

\vspace*{.5cm}

\noindent{\bf Acknowledgments.}  P. G. is supported in part by NSFC grant 11871082 and L. Z. by the the Faculty Silverstar Grant PS65447 at the University of New South Wales (UNSW). The authors are grateful to the anonymous referee for his/her very careful reading of the manuscript and many helpful suggestions.

\bibliography{biblio}
\bibliographystyle{amsxport}

\vspace*{.5cm}

\noindent\begin{tabular}{p{6cm}p{6cm}p{6cm}}
School of Mathematical Sciences & School of Mathematics and Statistics \\
Beihang University & University of New South Wales \\
Beijing 100191 China & Sydney NSW 2052 Australia \\
Email: {\tt penggao@buaa.edu.cn} & Email: {\tt l.zhao@unsw.edu.au} \\
\end{tabular}

\end{document}